\documentclass{siamltex}

\usepackage[latin9]{inputenc}
\usepackage{mathrsfs}
\usepackage{amsmath}
\usepackage{amssymb}
\usepackage{esint}
\usepackage{enumerate,url}

\usepackage[usenames]{color}

\newtheorem{remark}[theorem]{Remark}

\newcommand{\N}{\mathbb{N}}
\newcommand{\R}{\mathbb{R}}

\newcommand{\QQ}{\mathcal{Q}}
\newcommand{\LL}{\mathcal{L}}

\newcommand{\range}{\mathcal R}
\newcommand{\kernel}{\mathcal N}

\newcommand{\dx}[1][x]{\ensuremath{\,{\rm{d}} #1}}
\newcommand{\Langle}{\left\langle}
\newcommand{\Rangle}{\right\rangle}
\def\norm#1{\hspace{0.2ex} \|#1\| \hspace{0.2ex}} 
 
\newcommand{\kommentar}[1]{}

\begin{document}

\title{Monotonicity-based inversion of the fractional Schr\"odinger equation II. General potentials and stability}

\author{Bastian Harrach\footnotemark[2], and Yi-Hsuan Lin\footnotemark[3]}
\renewcommand{\thefootnote}{\fnsymbol{footnote}}

\footnotetext[2]{Institute for Mathematics, Goethe-University Frankfurt, Frankfurt am Main, 
Germany (harrach@math.uni-frankfurt.de)}

\footnotetext[3]{Department of Applied Mathematics, National Chiao Tung University, Hsinchu, Taiwan (yihsuanlin3@gmail.com)}

\maketitle

\let\thefootnote\relax\footnotetext{\hrule \vspace{1ex} \centering This is a preprint version of a journal article published in\\
 \emph{SIAM J. Math. Anal.} \textbf{52}(1), 402--436, 2020
(\url{https://doi.org/10.1137/19M1251576}).
}

\begin{abstract}
In this work, we use monotonicity-based methods for the fractional Schr\"odinger equation with general potentials $q\in L^\infty(\Omega)$
in a Lipschitz bounded open set $\Omega\subset \R^n$ in any dimension $n\in \N$. We demonstrate that if-and-only-if monotonicity relations between potentials and the Dirichlet-to-Neumann map hold up to a finite dimensional subspace. Based on these if-and-only-if monotonicity relations, we derive a constructive global uniqueness result for the fractional Calder\'on problem and its linearized version. We also derive a reconstruction method for unknown obstacles in a given domain that only requires the background solution of the fractional Schr\"odinger equation, and we prove uniqueness and Lipschitz stability from finitely many measurements for potentials lying in an a-priori known bounded set in a finite dimensional subset of $L^\infty(\Omega)$.
\end{abstract}



\section{Introduction }

Let $\Omega$ be a Lipschitz bounded open set in $\mathbb{R}^{n}$, $n\in \N$,
and $q\in L^{\infty}(\Omega)$ be a potential. For $0<s<1$, we consider the
Dirichlet problem for the nonlocal fractional Schr\"{o}dinger equation 
\begin{equation}\label{eq:intro_fractional}
(-\Delta)^{s}u+qu=0 \quad \text{ in }\Omega, \qquad u|_{\Omega_e}=F \quad \text{ in } \Omega_e:=\R^n\setminus \Omega,
\end{equation}
where the fractional Laplacian $(-\Delta)^{s}$ is defined by Fourier transform. We will consider the Calder\'on problem
of reconstructing an unknown potential $q$ from the Dirichlet-to-Neumann (DtN) operator
\[
\Lambda(q):\ H(\Omega_e)\to H(\Omega_e)^*, \quad F\mapsto (-\Delta)^s u|_{\Omega_e}, \quad \text{ where $u\in H^s(\R^n)$
solves \eqref{eq:intro_fractional},}
\]
cf.\ Section \ref{Section 2} for a precise definition of the DtN-operator and the function spaces, and 
\cite[Section 3]{ghosh2016calder} for further properties of the nonlocal DtN map $\Lambda_{q}$.

In the first part of this work \cite{harrach2019nonlocal-mono1}, we proved an if-and-only-if monotonicity relation 
between potentials $q\in L^\infty_+(\Omega)$ with positive essential infima and the associated DtN operators $\Lambda(q)$,
where the DtN operators are ordered in the sense of definiteness of quadratic forms (also known as Loewner order). From this relation, 
we obtained a constructive uniqueness result for the Calder\'on problem and a shape reconstruction method to determine unknown obstacles in a given domain.

The aim of this work is to drop the positivity assumption on the potential $q$ and extend the results from \cite{harrach2019nonlocal-mono1} to general 
potentials $q\in L^\infty(\Omega)$. Note that this may include resonant cases where $0$ is a Dirichlet eigenvalue 
of $(-\Delta)^{s}+q$ in $\Omega$. In such cases the Dirichlet problem \eqref{eq:intro_fractional} is only solvable 
in a subspace of the natural Dirichlet trace space $H(\Omega_e)$ with finite codimension, and the DtN operator $\Lambda(q)$ is defined accordingly,
cf.\ Section \ref{Section 2}. For general potentials $q_1,q_2\in L^\infty(\Omega)$, we will use a combination of monotonicity arguments and localized potentials to show that 
\[
q_1\leq q_2 \quad \text{ if and only if } \quad \Lambda(q_1)\leq_\text{fin} \Lambda(q_2),
\]
cf.\ Theorem \ref{Thm:if-and-only-if monotonicity}, where $q_1\leq q_2$ denotes that $q_1(x)\leq q_2(x)$ for almost every (a.e.) $x\in \Omega$, and
$\Lambda(q_1)\leq_\text{fin} \Lambda(q_2)$ denotes that the quadratic form associated with $\Lambda(q_2)-\Lambda(q_1)$ is non-negative 
on a subspace of $H(\Omega_e)$ with finite codimension (resp.\ on a subspace with finite codimension of the intersection of their domains of definition in the case of resonances).

This if-and-only-if monotonicity relation yields a constructive uniqueness proof for the fractional Calder\'on problem, cf.~Theorem~\ref{thm:constructive}.
For non-resonant potentials, we show a similar if-and-only-if monotonicity relation also for the linearized DtN-operators, and deduce uniqueness for the linearized Calder\'on problem, cf.\ Theorem \ref{thm:converse_mon_frechet}, and Corollary \ref{cor:Calderon_linearized}.

We then turn to the shape reconstruction (or inclusion detection) problem of locating regions where a unknown (non-resonant) coefficient function $q\in L^\infty(\Omega)$ differs from a known (non-resonant) reference function $q_0\in L^\infty(\Omega)$. 
We will show that this can be done without solving the fractional Schrödinger equation
for potentials other than the reference potentials $q_0$. In the indefinite case, with no further assumption on $q_0$ and $q$, we characterize the support of $q-q_0$ as the intersection of all closed sets fulfilling a linearized monotonicity condition, cf.\ Theorem~\ref{thm:support_from_closed_sets}. In the definite case, that either $q\geq q_0$ or $q_0\geq q$ in all of $\Omega$, we also obtain an easier characterization of  the (inner) support of $q-q_0$ as the union of all open balls fulfilling a linearized monotonicity condition, cf.\ Theorem~\ref{thm:support_from_open_balls}.

Our final result uses monotonicity and localized potentials arguments to show uniqueness and Lipschitz stability for the fractional Calder\'on problem with finitely many measurements for the case that the potential belongs to an a-priori known bounded set in a finite dimensional subset of $L^\infty(\Omega)$.

Let us give some references of the fast growing body of literature on inverse problems involving the non-local fractional Laplacian operator,
and relate our work to previous results. Fractional inverse problems appear when an imaging domain is investigated by an anomalous diffusion process and this process is more complicated than in the standard Brownian motion modeled by the Laplacian $-\Delta$. Global uniqueness for the Calder\'on problem for the fractional Schr\"odinger equation was first proven by Ghosh, Salo, and Uhlmann \cite{ghosh2016calder}, and 
the recent work of Ghosh, R\"uland, Salo, and Uhlmann \cite{ghosh2018uniqueness} shows uniqueness with a single measurement.
Note that both results rely on a very strong unique continuation property, and we will utilize this property from \cite{ghosh2016calder}
as a key ingredient for our results.
Furthermore, for uniqueness results, \cite{ghosh2017calder} and \cite{lai2018global} solved the Calder\'on problem for general nonlocal variable elliptic operators and the semilinear case, respectively. In addition, \cite{cekic2018calder} studied the fractional Calder\'on problem with drift, which shows the global uniqueness result holds for drift and potential simultaneously, which is the first example to demonstrate different results between local and nonlocal inverse problems. Recently, \cite{LLR2019Calderonparabolic} investigated the Calder\'on problem for a space-time fractional parabolic equation. We also refer readers to \cite{cao2017simultaneously,cao2018determining} for further studies on the simultaneous determination of parameters in fractional inverse problems.

Arguments combining PDE-based estimates with blow-up techniques have a long history in the study of inverse coefficients problems, see, e.g., \cite{alessandrini1990,ikehata1999identification,isakov1988,kohn1984determining,kohn1985determining}.
The technique of combining monotonicity estimates with localized potentials \cite{gebauer2008localized} as used herein is a flexible recent approach that has already lead to a number of results, cf.\ \cite{arnold2013unique,barth2017detecting,brander2018monotonicity,griesmaier2018monotonicity,harrach2009uniqueness,harrach2012simultaneous,harrach2019nonlocal-mono1,harrach2018localizing,harrach2019dimension,harrach2018helmholtz,harrach2010exact,harrach2013monotonicity,harrach2017local,seo2018learning}. Also, several recent works build practical reconstruction methods on monotonicity properties \cite{furuya2019monotonicity,garde2017comparison,garde2019reconstruction,garde2017convergence,garde2019regularized,harrach2015combining,harrach2016enhancing,harrach2018monotonicity,harrach2015resolution,maffucci2016novel,su2017monotonicity,tamburrino2002new,tamburrino2016monotonicity,ventre2017design,zhou2018monotonicity}.
Notably, the present work shows that monotonicity-based reconstruction methods that have been developed for standard diffusion processes
can also be applied to the fractional diffusion case and that the methods even become simpler and more powerful due to the very strong
unique continuation property of Ghosh, Salo, and Uhlmann \cite{ghosh2016calder}. Moreover, we
derive in this work a new result on the existence of simultaneously localized potentials for two coefficient functions, that may 
be of importance also in the study of other inverse problems.

 Logarithmic stability results for the fractional Schr\"odinger equation and their optimality were proven by Rüland and Salo in \cite{ruland2017stability,ruland2018exponential}.
Lipschitz stability for the finite dimensional fractional Calder\'on problem with a specific set of finitely many measurements (that depend on the unknown potentials) was shown by R\"uland and Sincich in \cite{ruland2018lipschitz}. Note that our Lipschitz stability result in Section \ref{Section 5} complements the result in \cite{ruland2018lipschitz} as we show that any sufficiently high number of measurements
(depending only on the a-priori data but not on the unknown potentials) uniquely determines the potential 
and that Lipschitz stability holds. Moreover, let us stress that the idea of using monotonicity and localized potentials arguments for proving Lipschitz stability (that was already utilized in \cite{eberle2019lipschitz,harrach2019uniqueness,harrach2019global,seo2018learning}), differs from traditional approaches that are mostly based on quantitative unique continuation or quantitative Runge approximation, cf.,
\cite{alessandrini1996determining,alessandrini2018lipschitz,alessandrini2017lipschitz,alessandrini2005lipschitz,bacchelli2006lipschitz,beilina2017lipschitz,bellassoued2006lipschitz,bellassoued2007lipschitz,beretta2017uniqueness,beretta2013lipschitz,beretta2011lipschitz,cheng2003lipschitz,imanuvilov1998lipschitz,imanuvilov2001global,kazemi1993stability,klibanov2006lipschitz_nonstandard,klibanov2006lipschitz,melendez2013lipschitz,ruland2018lipschitz,sincich2007lipschitz,yuan2007lipschitz,yuan2009lipschitz}.
Our new approach of showing Lipschitz stability seems conceptually simpler as it does not require quantitative analytic estimates. On the downside, our new approach does not give any analytic bounds on the Lipschitz stability constants that may characterize the asymptotic instability when the dimension of the ansatz space tends to infinity. It may however, lead to a numerical algorithm to calculate the Lipschitz constant for a given setting, cf.~\cite{harrach2019stable,harrach2019global}, which might be important to
quantify the achievable resolution and noise robustness in practical applications.

The main technical difficulty in extending the results from the positive potentials case \cite{harrach2019nonlocal-mono1} to general coefficients $q\in L^\infty(\Omega)$ is to prove two new extensions
of the localized potentials approach \cite{gebauer2008localized}. 
For general potentials, the variational formulation of the fractional Schr\"odinger equation is no longer coercive but a compact perturbation of a coercive formulation and resonances may arise. To overcome this difficulty, we
use an approach that originated in \cite{harrach2018helmholtz} and work in spaces of finite codimension where the formulation is still coercive and resonances are excluded. This makes it necessary to prove that any subspace of finite codimension contains localized potentials.
The second major difficulty comes from the fact that only the simpler monotonicity inequality in \cite[Lemma~3.1]{harrach2019nonlocal-mono1} can be extended to general potentials, cf.\ Theorem~\ref{Theorem for monotonicity} in this work. This makes it necessary to prove that
localized potentials exist for two different coefficients simultaneously (and in any subspace of finite codimension). It can be expected that the idea of simultaneously localized potentials introduced in this work will also be helpful to extend monotonicity-based methods to other applications.

The paper is structured as follows. In Section \ref{Section 2}, we summarize the variational theory for the fractional Schr\"odinger equation, introduce the DtN operator and the unique continuation property from \cite{ghosh2016calder}. In Section 3, we define a generalized Loewner order for linear operators, which holds up to a finite dimensional subspace of a Hilbert space. We also show that increasing potentials $q$ monotonically increases the corresponding DtN map $\Lambda_q$ in the sense of this generalized Loewner order, 
and prove the existence of localized potentials to control the energy terms appearing in the monotonicity relations.
The last two sections contain our main results. 
In Section \ref{Section 4}, we investigate a converse result for the monotonicity relations 
using localized potentials, to deduce if-and-only-if monotonicity relations between the DtN map and the potentials. Based on these results, we prove uniqueness for the fractional Calder\'on problem in a constructive way. 
We also prove uniqueness for the linearized fractional Calder\'on problem and develop an inclusion detection algorithm based on monotonicity tests. Finally, in Section \ref{Section 5}, we use the monotonicity relations and the localized potentials, to prove uniqueness and Lipschitz stability in finite dimensional subspaces by finitely many measurements.

\section{The fractional Schr\"{o}dinger equation for general potentials \label{Section 2}}

Throughout this work let $s\in (0,1)$, $n\in \N$, $\Omega\subseteq \R^n$ be a Lipschitz bounded open set,
and $q\in L^\infty(\Omega)$. All function spaces in this work are real-valued unless indicated otherwise. In this section, we briefly summarize some notations and results on the fractional Schr\"{o}dinger equation and the associated Dirichlet problem.

\subsection{Variational formulation of the fractional Schr\"{o}dinger equation}

As in \cite{harrach2019nonlocal-mono1} we consider the fractional Laplacian (defined by Fourier transform) as an operator 
\[
(-\Delta)^{s}:\ L^2(\mathbb{R}^{n})\to \mathcal{S}'(\mathbb{R}^{n}),
\]
The fractional Sobolev space is defined by 
\[
H^{s}(\mathbb{R}^{n}):=\{u\in L^{2}(\mathbb{R}^{n}):\ (-\Delta)^{s/2}u\in L^{2}(\mathbb{R}^{n})\}
\]
and equipped with the scalar product 
\[
\left(u,v\right)_{H^{s}(\mathbb{R}^{n})}:=\int_{\mathbb{R}^{n}}\left((-\Delta)^{s/2}u\cdot(-\Delta)^{s/2}v+uv\right)\dx
\quad \text{ for all }u,v\in H^{s}(\mathbb{R}^{n}).
\]
It can be shown that $H^{s}(\mathbb{R}^{n})$ is a Hilbert space, cf., e.g., \cite{di2012hitchhiks}. Let
\[
H_{0}^{s}(\Omega):=\text{closure of \ensuremath{C_{c}^{\infty}(\Omega)} in \ensuremath{H^{s}(\mathbb{R}^{n})}},
\]
and note that this space is sometimes denoted as $\widetilde{H}^{s}(\Omega)$ in the literature, e.g., \cite{ghosh2016calder,ghosh2017calder}.

We also define the bilinear form 
\begin{equation*}
\mathscr{B}_{q}(u,w):=\int_{\mathbb{R}^{n}}(-\Delta)^{s/2}u\cdot(-\Delta)^{s/2}w\dx+\int_{\Omega}quw\,\dx
\quad \text{ for } u,w\in H^{s}(\mathbb{R}^{n}).
\end{equation*}
Then, for any $f\in L^{2}(\Omega)$, $u\in H^{s}(\mathbb{R}^{n})$ solves (in the sense of distributions)
\[
(-\Delta)^{s}u+qu=f\quad\text{ in \ensuremath{\Omega}}
\]
if and only if $u\in H^{s}(\mathbb{R}^{n})$ fulfills the variational formulation
\begin{equation}\label{eq:VarFormFractional}
\mathscr{B}_{q}(u,w)=\int_{\Omega}fw\dx \quad\text{ for all }w\in H_{0}^{s}(\Omega),
\end{equation}
cf., e.g., \cite[Lemma~2.1]{harrach2019nonlocal-mono1}.

\subsection{The Dirichlet boundary value problem}\label{subsect:Dirichlet}

The Dirichlet trace operator on $\Omega_e:=\R^n\setminus \overline \Omega$ can be defined using abstract quotient spaces by setting
\[
\gamma_{\Omega_{e}}^{(D)}:\ H^{s}(\mathbb{R}^{n})\to H(\Omega_{e}):=H^{s}(\mathbb{R}^{n})/H_{0}^{s}(\Omega),\quad u\mapsto u+H_{0}^{s}(\Omega).
\]
Then, by definition, $\gamma_{\Omega_{e}}^{(D)}$ is surjective, 
$H_0^s(\Omega)=\{u\in H^{s}(\mathbb{R}^{n}):\ \gamma_{\Omega_{e}}^{(D)}u=0\}$. Moreover, for all $u,v\in H^{s}(\mathbb R^n)$, 
\begin{equation}\label{eq:Dirichlet_trace}
\gamma_{\Omega_{e}}^{(D)}u=\gamma_{\Omega_{e}}^{(D)}v\quad\text{ implies that }\quad u(x)=v(x)\quad\text{ for \ensuremath{x\in\Omega_{e}} a.e.,}
\end{equation}
cf., e.g., \cite[Lemma 2.2]{harrach2019nonlocal-mono1}. This implies that $\gamma_{\Omega_{e}}^{(D)}$ is an injective mapping from $C_c^\infty(\Omega_e)$
into $H(\Omega_{e})$. For the sake of readability we will write $u|_{\Omega_{e}}$ instead
of $\gamma^{(D)}_{\Omega_{e}}u$ throughout this work, and identify $C_c^\infty(\Omega_e)$ with its
image in $H(\Omega_{e})$. 

Throughout this work, we will use that for all $u,w\in H_0^s(\Omega)$
\[
\mathscr{B}_{q}(u,w)=\left( (I-\iota^*\iota +\iota^* M_q \iota)u,w\right)_{H^s_0(\Omega)}
\]
with the bounded linear operators
\begin{align*}
    I: & \ H_0^s(\Omega) \to H_0^s(\Omega),\\
    \iota: &\ H_0^s(\Omega)\to L^2(\Omega),\\
    M_{q}: & \ L^2(\Omega)\to L^2(\Omega),
\end{align*}
denoting the identity operator, the compact restriction and embedding,
cf.~\cite[Lemma 10]{palatucci2013local}, and the multiplication operator by $q$. 

We then have the following result on the solvability of the Dirichlet boundary value problem.
\begin{lemma}\label{lemma:solvability_Dirichlet_bvp}
Let $F\in H(\Omega_{e})$, $f\in L^{2}(\Omega)$, and
\[
N_{q}:=\{ u\in H_0^s(\Omega):\ (-\Delta)^{s}u+qu=0\quad\text{ in $\Omega$} \}.
\]
\begin{enumerate}[(a)]
\item $u\in H^{s}(\mathbb{R}^{n})$ solves the Dirichlet problem
\begin{align}\label{eq:Dirichlet_with_rhs}
(-\Delta)^{s}u+qu=f\quad\text{ in \ensuremath{\Omega}},\quad u|_{\Omega_{e}}=F,
\end{align}
if and only if $u=u^{(0)}+u^{(F)}$, where $u^{(F)}\in H^{s}(\mathbb{R}^{n})$
fulfills $u^{(F)}|_{\Omega_{e}}=F$, 
and $u^{(0)}\in H_{0}^{s}(\Omega)$ solves 
\[
\mathscr{B}_{q}(u^{(0)},w)=-\mathscr{B}_{q}(u^{(F)},w)+\int_{\Omega}fw\dx\quad\text{ for all }w\in H_{0}^{s}(\Omega).
\]

Note that for $F\in C_{c}^{\infty}(\Omega_{e})$ one can simply choose $u^{(F)}:=F$.
\item $N_{q}$ is finite-dimensional. The Dirichlet problem \eqref{eq:Dirichlet_with_rhs}
is solvable if and only if 
\begin{equation}\label{eq:solvability_condition}
\mathscr{B}_{q}(u^{(F)},w) = \int_{\Omega}fw\dx  \quad \text{ for all $w\in N_q$.}
\end{equation}
The solution $u\in H^{s}(\mathbb{R}^{n})$ of \eqref{eq:Dirichlet_with_rhs} is unique up to addition of a function in $N_q$,
and $u+N_q\in H^{s}(\mathbb{R}^{n}) / N_q$ depends linearly and continuously on $F\in H(\Omega_{e})$ and $f\in L^{2}(\Omega)$.
\end{enumerate}
\end{lemma}
\begin{proof}
(a) immediately follows from the variational formulation \eqref{eq:VarFormFractional}.

To prove (b), we use the Riesz representation theorem to obtain $v_f^F\in H_0^s(\Omega)$ fulfilling
\[
\left( v_f^F,w \right)_{H_0^s(\Omega)} = -\mathscr{B}_{q}(u^{(F)},w)+\int_{\Omega}fw\dx\quad\text{ for all }w\in H_{0}^{s}(\Omega).
\]

Using (a), and that $w\in H_0^s(\Omega)$ implies 
$w(x)=0$ for $x\in\Omega_{e}$ a.e., we obtain that $u\in H^{s}(\mathbb{R}^{n})$ solves \eqref{eq:Dirichlet_with_rhs} if and only if 
$u=u^{(0)}+u^{(F)}$ with $u^{(0)}\in H_{0}^{s}(\Omega)$ solving
\begin{align*}
\lefteqn{\quad \left( ( I - \iota^*\iota+\iota^*M_q \iota)u^{(0)},w\right)_{H_0^s(\Omega)}}\\
&= \mathscr{B}_{q}(u^{(0)},w)=-\mathscr{B}_{q}(u^{(F)},w)+\int_{\Omega}fw\dx
= \left( v_f^F,w \right)_{H_0^s(\Omega)} \quad\text{ for all }w\in H_{0}^{s}(\Omega),
\end{align*}
i.e. 
\[
 ( I - \iota^*\iota+\iota^*M_q \iota)u^{(0)} = v_f^F,
\]
and that 
\begin{align}\label{N_q}
	N_q= \kernel (I - \iota^*\iota+\iota^*M_q \iota).
\end{align} 
Here $\kernel(A)$ stands for the kernel of the linear operator $A$. Since $\iota^*\iota-\iota^*M_q \iota$ is compact and self-adjoint, Fredholm theory
(cf., e.g., \cite[Appendix D, Theorem~5]{evans2010partial}) yields that $N_q$ is finite-dimensional, and that 
\eqref{eq:Dirichlet_with_rhs} is solvable if and only if 
\[
\left( v_f^F ,  w \right)_{H_0^s(\Omega)} = 0 \quad \text{ for all $w\in \kernel (I - \iota^*\iota+\iota^*M_q \iota)=N_q$,}
\]
which gives the condition \eqref{eq:solvability_condition}. 

Clearly $u^{(0)}$ is unique up to addition of a function in $N_q$,
and $u^{(0)}+N_q$ depends linearly and continuously on $v_f^F\in H_0^s(\Omega)$. It
easily follows that $u=u^{(0)}+u^{(F)}$ is unique up to addition of a function in $N_q$,
and that $u+N_q\in H^{s}(\mathbb{R}^{n}) / N_q$ depends linearly and continuously on $F\in H(\Omega_{e})$ and $f\in L^{2}(\Omega)$.
\end{proof}

\begin{corollary}\label{cor:Dirichlet_bvp}
Let $H_q^s(\R^n)\subseteq H^s(\R^n)$ be the $H^s(\R^n)$-orthogonal complement of $N_q$, and 
\begin{align*}
H_q(\Omega_e):=\{ F\in H(\Omega_e):\ \mathscr{B}_{q}(u^{(F)},w) = 0  \quad \text{ for all $w\in N_q$}\}.
\end{align*}
Then the codimension of $H_q(\Omega_e)$ in $H(\Omega_e)$ is at most $\dim N_q$, and for all $F\in H_q(\Omega_e)$ there exists a unique solution $u\in H_q^s(\R^n)$ of the Dirichlet problem
\begin{align}\label{eq:Dirichlet_bvp_Solution_operator}
(-\Delta)^{s}u+qu=0\quad\text{ in \ensuremath{\Omega}},\quad u|_{\Omega_{e}}=F,
\end{align}
and that the solution operator
\[
S_q:\ H_q(\Omega_e)\to H_q^s(\R^n),\quad F \mapsto u, \quad \text{where $u$ solves \eqref{eq:Dirichlet_bvp_Solution_operator}},
\]
is linear and bounded.
\end{corollary}
\begin{proof}
We first show that $H_q(\Omega_e)$ is well-defined. If $u^{(F)},\widetilde u^{(F)} \in H^s(\R^n)$ both fulfill $u^{(F)}|_{\Omega_e}=F=\widetilde u^{(F)}|_{\Omega_e}$, then $u^{(F)}-\widetilde u^{(F)}|_{\Omega_e}\in H_0^s(\Omega)$ and thus it follows from the definition of
$N_q$ \eqref{N_q} and \eqref{eq:VarFormFractional} that
\[
\mathscr{B}_{q}(u^{(F)}-\widetilde u^{(F)},w)=0 \quad \text{ for all } w\in N_q.
\]

Next, we show that the codimension of $H_q(\Omega_e)$ in $H(\Omega_e)$ is at most $d:=\dim N_q$. Let $(w_1,\ldots,w_d)\subset N_q$ be an 
orthonormal basis of $N_q$, and let $\gamma^{-}:\ H(\Omega_e)\to H^s(\R^n)$ be a linear right inverse of the Dirichlet trace operator $\gamma_{\Omega_{e}}^{(D)}$. Then, by linearity,
\begin{align*}
H_q(\Omega_e)=\{ F\in H(\Omega_e):\ \mathscr{B}_{q}(\gamma^{-} F,w_j) = 0  \quad \text{ for all $j=1,\ldots,d$}\}
=\kernel(\mathcal A),
\end{align*}
with a linear operator
\[
\mathcal A:\ H(\Omega_e)\to \R^d,\quad  F\mapsto \left( \mathscr{B}_{q}(\gamma^{-} F,w_j)  \right)_{j=1,\ldots,d}.
\]
Hence, the codimension of $H_q(\Omega_e)=\kernel(\mathcal A)$ is $\dim(\range(\mathcal A))\leq d$.

Finally, it follows from Lemma~\ref{lemma:solvability_Dirichlet_bvp}(b) that \eqref{eq:Dirichlet_bvp_Solution_operator}
possesses a solution $\widetilde u\in H^s(\R^n)$ which is unique up to addition of a function in $N_q$. Hence,
\[
u:=\widetilde u - \sum_{j=1}^d  w_j (\widetilde u,w_j)_{H^s(\R^n)} \in H_q^s(\R^n)
\]
solves \eqref{eq:Dirichlet_bvp_Solution_operator}, and $H_q^s(\R^n)$ contains no other solutions
of \eqref{eq:Dirichlet_bvp_Solution_operator}. Since $H_q^s(\R^n)$ is isomorphic to $H^s(\R^n)/N_q$,
the continuity and linearity of the solution operator $S_q$ also follow from Lemma~\ref{lemma:solvability_Dirichlet_bvp}(b). 
\end{proof}

\subsection{Neumann traces and the Dirichlet-to-Neumann operator}

We define the Neumann trace operator 
\[
\gamma_{\Omega_e}^{(N)}:\ H^{s}_\Delta(\mathbb{R}^{n}):=
\left\{ u\in H^{s}(\mathbb{R}^{n}):\ \exists f\in L^2(\Omega) \text{ with } (-\Delta)^{s}u=f\text{ in $\Omega$}\right\}
\to H(\Omega_{e})^{*}
\]
by setting
\begin{equation}\label{eq:Neumanntrace}
\left\langle \gamma_{\Omega_e}^{(N)}u,F\right\rangle :=\int_{\R^n} (-\Delta)^{s/2} u \cdot (-\Delta)^{s/2} v^{(F)} \dx
- \int_\Omega (-\Delta)^{s}u \cdot v^{(F)} \dx,
\end{equation}
where $v^{(F)}\in H^{s}(\mathbb{R}^{n})$ fulfills $v^{(F)}|_{\Omega_{e}}=F$, $H(\Omega_e)^*$ is the dual space of $H(\Omega_e)$, and throughout this paper $\left\langle \cdot,\cdot\right\rangle$ denotes the dual pairing on $H(\Omega_{e})^{*}\times H(\Omega_{e})$.
Note that $\gamma_{\Omega_e}^{(N)}u$ is well-defined since the right hand side of \eqref{eq:Neumanntrace} does not depend on the choice of $v^{(F)}$,
and that $\gamma_{\Omega_e}^{(N)}$ is a bounded linear operator.

For the sake of readability, we also use the formal notation $(-\Delta)^{s}u|_{\Omega_{e}}:=\gamma_{\Omega_e}^{(N)}u$ 
for the Neumann trace, which can be motivated by the following lemma, see also \cite[Remark 2.4]{harrach2019nonlocal-mono1} and \cite{ghosh2016calder} for further justifications of this notation under additional smoothness conditions on $u$ or $\Omega$. 

\begin{lemma}\label{lemma:Neumann_trace}
Let $u\in H^{s}_\Delta(\mathbb{R}^{n})$. If $\gamma_{\Omega_e}^{(N)}u\in L^2(\Omega)$ in the sense that there exists $g\in L^2(\Omega_e)$ with
\[
\left\langle \gamma_{\Omega_e}^{(N)}u,F\right\rangle=\int_{\Omega_e} g v^{(F)} \dx \quad \text{ for all }
v^{(F)}\in H^s(\R^n) \text{ with } v^{(F)}|_{\Omega_e}=F,
\]
then 
$g=(-\Delta)^s u$ in $\Omega_e$ (in the sense of distributions).
\end{lemma}
\begin{proof}
For all $\varphi\in C_c^\infty(\Omega_e)\subseteq H(\Omega_e)$ (cf.\ subsection \ref{subsect:Dirichlet}), we have that
\begin{align*}
\int_\Omega g \varphi \dx&= \left\langle \gamma_{\Omega_e}^{(N)}u,G\right\rangle
= \int_{\R^n} (-\Delta)^{s/2} u \cdot (-\Delta)^{s/2} \varphi \dx
- \int_\Omega (-\Delta)^{s}u \cdot \varphi \dx\\
&= \int_{\R^n} (-\Delta)^{s/2} u \cdot (-\Delta)^{s/2} \varphi \dx
= \langle (-\Delta)^s u, \varphi\rangle_{\mathcal D'(\Omega_e)\times \mathcal D(\Omega_e)}.
\end{align*}
\end{proof}

Note also that if 
$u\in H^s(\R^n)$ solves $(-\Delta)^{s}u+qu=0$ in $\Omega$, then 
\[
\left\langle (-\Delta)^{s}u|_{\Omega_{e}},G\right\rangle =\mathscr{B}_{q}(u,v^{(G)})
\]
holds for all $G\in H(\Omega_e)$ and all $v^{(G)}\in H^s(\R^n)$ with $v^{(G)}|_{\Omega_e}=G$.
Using Corollary~\ref{cor:Dirichlet_bvp}, we can thus define the linear bounded DtN operator
\[
\Lambda(q):\ H_q(\Omega_e)\to H(\Omega_e)^*,\quad F\mapsto (-\Delta)^{s}u|_{\Omega_{e}}
\]
where $u\in H_q^s(\R^n)$ solves 
\begin{align*}
(-\Delta)^{s}u+qu=0\quad\text{ in \ensuremath{\Omega}},\quad u|_{\Omega_{e}}=F.
\end{align*}

In view of the following sections, note that for $q_1,q_2\in L^\infty(\Omega)$, 
\[
H_{q_1,q_2}(\Omega_e):=H_{q_1}(\Omega_e)\cap H_{q_2}(\Omega_e)
\]
is a subspace of $H(\Omega_e)$ with codimension less than or equal to
$\dim N_{q_1}+\dim N_{q_2}$, on which both $\Lambda(q_1)$ and $\Lambda(q_2)$ are defined. Hence, throughout this work,
$\Lambda(q_1)-\Lambda(q_2)$ will always denote the linear bounded operator
\[
\Lambda(q_1)-\Lambda(q_2):\ H_{q_1,q_2}(\Omega_e)\to H(\Omega_e)^*.
\]

The following relation between the DtN operator and the bilinear form will be useful.
\begin{lemma}\label{lemma:NtD_and_bilinear_forms}
Let $q_1,q_2\in L^\infty(\Omega)$, $F\in H_{q_1}(\Omega_e)$, $G\in H_{q_2}(\Omega_e)$, and let 
$u\in H_{q_1}^s(\R^n)$, $v\in H_{q_2}^s(\R^n)$ solve 
\begin{align*}
(-\Delta)^{s}u+q_1 u=0\quad\text{ in \ensuremath{\Omega}},\quad u|_{\Omega_{e}}=F,\\
(-\Delta)^{s}v+q_2 v=0\quad\text{ in \ensuremath{\Omega}},\quad v|_{\Omega_{e}}=G.
\end{align*}
Then
\[
\left\langle \Lambda(q_1) F, F \right\rangle= \mathscr{B}_{q_1}(u,u) \quad \text{ and } \quad
\left\langle \Lambda(q_1) F, G \right\rangle= \mathscr{B}_{q_1}(u,v),
\]
and under the additional restriction that $F,G\in H_{q_1,q_2}(\Omega_e)$ this also implies that
\[
\left\langle \left( \Lambda(q_1) - \Lambda(q_2)\right) F, G \right\rangle= \mathscr{B}_{q_1}(u,v)-\mathscr{B}_{q_2}(u,v)=\int_\Omega (q_1-q_2)uv\dx.
\]
\end{lemma}
\begin{proof}
This immediately follows from the variational formulation in Lemma~\ref{lemma:solvability_Dirichlet_bvp} and the definition of the Neumann trace.
\end{proof}

\subsection{Unique continuation from open sets and Cauchy data}\label{subsect:UCP}

We recall the unique continuation result from Ghosh, Salo and Uhlmann \cite{ghosh2016calder}:

\begin{theorem}\cite[Theorem 1.2]{ghosh2016calder} \label{thm:unique_continuation}
Let $n\in \N$, and $0<s<1$. If $u\in H^{r}(\mathbb{R}^{n})$ for
some $r\in\mathbb{R}$, and both $u$ and $(-\Delta)^{s}u$ vanish
in the same arbitrary non-empty open set in $\mathbb R^n$, then $u\equiv 0$ in $\mathbb{R}^{n}$. 
\end{theorem}

We will make use of the following simple corollary.
\begin{corollary}\label{cor:UCP}
Let $u\in H^s(\R^n)$ solve $(-\Delta)^s u + qu =f$ in $\Omega$, with $f\in L^2(\Omega)$
\begin{enumerate}[(a)]
\item If $u$ and $f$ vanish in the same nonempty open set $\mathcal O \subset \Omega$, then $u\equiv 0 $ in $\mathbb R^n$.
\item If $u|_{\Omega_e}=0$ and $(-\Delta)^{s}u|_{\Omega_{e}}=0$, then $u\equiv 0$ in $\mathbb R^n$.
\end{enumerate}
\end{corollary}
\begin{proof}
(a) follows since $u=0$ in $\mathcal O$, and $(-\Delta)^s u + qu =0$ in $\mathcal O$, implies 
$(-\Delta)^s u=0$ in $\mathcal O$. For (b) note that $u|_{\Omega_e}$ and $(-\Delta)^{s}u|_{\Omega_{e}}$ are only formal notations
for the Dirichlet and Neumann traces of $u$, but 
$u|_{\Omega_e}=0$, and $(-\Delta)^{s}u|_{\Omega_{e}}=0$
do imply that
\[
u=0 \quad \text{ in $\Omega_e$,} \quad \text{ and } \quad (-\Delta)^{s}u=0\quad \text{ in $\Omega_e$}
\]
in the sense of distributions by \eqref{eq:Dirichlet_trace} and Lemma \ref{lemma:Neumann_trace}. Hence, both cases follow from Theorem \ref{thm:unique_continuation}.
\end{proof}

\begin{remark}
When $\frac{1}{4}\leq s <1$, then the unique continuation property in Corollary~\ref{cor:UCP}(a) already holds under the weaker condition
that $u$ vanishes in a subset of $\Omega$ with positive measure, cf.\ \cite[Proposition 5.1]{ghosh2018uniqueness}.
Moreover, based on such property, \cite{ghosh2018uniqueness} shows global uniqueness for the fractional Schr\"odinger equation by a single measurement.
\end{remark}

\section{Monotonicity relations and localized potentials\label{Section 3}}

In this section we derive monotonicity relations between $L^\infty(\Omega)$ potentials and their associated DtN operators,
and show how to control the energy terms in the monotonicity relations with the technique of localized potentials.

\subsection{Monotonicity relations}

We characterize the monotonicity relations between DtN operators with an extended \emph{Loewner order} 
that holds up to finite dimensional subspaces. 

\begin{definition}\label{def:generalized_Loewner}
Let $H$ be a Hilbert space and $H_1,H_2\subseteq H$ be two subspaces of finite codimension, and let 
$L_1:\ H_1\to H$, $L_2:\ H_2\to H$ be two linear bounded operators.
For a number $d\in \N_0:=\N \cup\{0\}$ we write
\[
L_1\leq_{d} L_2 
\]
if there exists a subspace $W\subseteq H_{12}:=H_1\cap H_2$ with $\dim(W)\leq d$, and
\[
\langle (L_2-L_1)v,v\rangle \geq 0\quad \text{ for all } \quad v\in W^\perp\subseteq H_{12}.
\]
Here and in the following, we use the notation $W^\perp\subseteq H_{12}$ to indicate that the orthogonal complement is taken in $H_{12}$.

We write $L_1\leq L_2$ if $L_1\leq_{0} L_2$, and $L_1\leq_{\text{fin}} L_2$ if $L_1\leq_{d} L_2$ for some $d\in \N_0$.
We also write
\[
L_1\stackrel{\text{fin}}{=} L_2 \quad \text{ if } \quad L_1\leq_{\text{fin}} L_2, \quad \text{ and } \quad L_2\leq_{\text{fin}} L_1,
\]
i.e. if there exists a finite dimensional subspace $W\subseteq H_{12}$ so that 
\[
\langle (L_2-L_1)v,v\rangle = 0 \quad \text{ for all } v\in W^\perp\subseteq H_{12}.
\]
\end{definition}

Note that if $H_1=H_2=H$ and $L_1,L_2$ are self-adjoint and compact, this is the same extended Loewner order as in \cite{harrach2018helmholtz}. 

Let us stress that the binary relation $\leq_d$ is reflexive, but generally neither transitive, nor antisymmetric.
Obviously, $L_1\leq_{d_1} L_2$ and $L_2\leq_{d_2} L_3$ imply that $L_1\leq_{d} L_3$, with $d=d_1+d_2+\mathrm{codim}(H_2)$, so that $\leq_{\text{fin}}$ is 
a reflexive and transitive relation, i.e., a preorder. Moreover, Corollaries~\ref{cor:Calderon} and \ref{cor:Calderon_linearized} will show that $\leq_{\text{fin}}$ is antisymmetric on the set of NtD operators and on their linearizations around a fixed non-resonant potential, so that on these sets, $\leq_{\text{fin}}$ is a partial order. 

For two potentials $q_1,q_2\in L^\infty(\Omega)$ we write $q_1\leq q_2$ if $q_1(x)\leq q_2(x)$ for almost everywhere (a.e.) $x\in \Omega$.
We will show that increasing the potential $q$ in this sense increases
the DtN map $\Lambda(q)$ in the sense of the generalized Loewner order in Definition \ref{def:generalized_Loewner}. 
Note that monotonicity relations in inverse coefficient problems go back to the works of Ikehata, Kang, Seo, and Sheen \cite{ikehata1998size,kang1997inverse},
and they have been at the core of many reconstruction algorithms including the Factorization method and the Monotonicity method, cf.\ the list of references in the introduction. Extensions of monotonicity relations to subspaces of finite codimensions have first been studied in \cite{harrach2018helmholtz,griesmaier2018monotonicity}, and we follow the general approach from there. A sharper bound on the 
dimension of the excluded subspaces has recently been obtained for the standard Helmholtz equation in \cite{harrach2019dimension}.

\begin{definition}\label{def:dimension_d}
For $q\in L^\infty(\Omega)$ let $d(q)\in \N_0$ denote the number of eigenvalues (counted with multiplicity) of the
compact self-adjoint operator $\iota^* \iota - \iota^* M_q \iota$ that are greater than $1$.
\end{definition}

\begin{theorem}[Monotonicity relations]\label{Theorem for monotonicity} 
Let $q_{1},q_{2}\in L^{\infty}(\Omega)$. There exists a subspace $V\subseteq H_{q_1,q_2}(\Omega_e)$
with $\dim(V)\leq d(q_2)$ so that 
\begin{align}\label{eq:monotonicity}
\Langle \left(\Lambda(q_1)-\Lambda(q_2)\right) F, F \Rangle 
\geq \int_\Omega (q_1-q_2)|u_1|^2 \dx \quad \text{ for all } F\in V^\perp\subseteq H_{q_1,q_2}(\Omega_e),
\end{align}
where $u_1\in H_{q_1}^s(\R^n)$ solves 
  $(-\Delta)^{s} u_1+q_1 u_1=0$ in $\Omega$ with $u_1|_{\Omega_{e}}=F$.

Hence
\begin{align*}
q_1\geq q_2 \text{ a.e. in }\Omega \quad  \text{ implies that } \quad \Lambda(q_1)\geq_{d(q_2)}\Lambda(q_2).
\end{align*}
\end{theorem}

Before we prove Theorem~\ref{Theorem for monotonicity}, let us 
also formulate a variant that will be useful for applying the idea of localized potentials in the next sections,
remark on interchanging $q_1$ and $q_2$, and discuss the dependence of $\dim(N_q)$ and $d(q)$ on $q$. 

\begin{theorem}\label{thm:useful} 
Let $q_{1},q_{2}\in L^{\infty}(\Omega)$. There exists a subspace 
\[
V_+\subseteq H_{q_1,q_2}(\Omega_e)\quad \text{ with } \quad \dim(V_+)\leq d(q_2)+\dim(N_{q_2}),
\]
and a constant $\lambda>0$, so that for all $F\in V_+^\perp\subseteq H_{q_1,q_2}(\Omega_e)$
\begin{align}\label{eq:useful_monotonicity}
\Langle \left(\Lambda(q_1)-\Lambda(q_2)\right) F, F \Rangle 
\geq \int_\Omega (q_1-q_2)|u_1|^2 \dx + \lambda\norm{u_1-u_2}_{H^s(\R^n)}^2
\end{align}
and, for all $D\subseteq \Omega$ containing $\supp(q_1-q_2)$,
\begin{align}\label{eq:useful_bound}
\norm{u_2}_{L^2(D)}\leq c \norm{u_1}_{L^2(D)},
\end{align}
where $c:=1+\frac{1}{\lambda}\norm{q_1-q_2}_{L^\infty(D)}$, and, for $j=1,2$,
$u_j\in H_{q_j}^s(\R^n)$ solve
\[
(-\Delta)^{s} u_j+q_j u_j=0\quad \text{ in } \Omega,\qquad u_j|_{\Omega_{e}}=F.
\]
\end{theorem}

\begin{remark}\label{remark_monotonicity}
By interchanging $q_1$ and $q_2$ in Theorems \ref{Theorem for monotonicity} and \ref{thm:useful}, we also obtain that there exist 
 subspaces 
\[
V,V_+\subseteq H_{q_1,q_2}(\Omega_e)\quad \text{ with } \quad \dim(V)\leq d(q_1), \text{ and } \dim(V_+)\leq d(q_1)+\dim(N_{q_1}),
\]
and a constant $\lambda>0$, so that 
\begin{alignat*}{2}
\Langle \left(\Lambda(q_1)-\Lambda(q_2)\right) F, F \Rangle 
&\leq \int_\Omega (q_1-q_2)|u_2|^2 \dx \quad && \text{ for all } F\in V^\perp\subseteq H_{q_1,q_2}(\Omega_e),
\end{alignat*}
and
\begin{align*}
\Langle \left(\Lambda(q_1)-\Lambda(q_2)\right) F, F \Rangle 
&\leq \int_\Omega (q_1-q_2)|u_2|^2 \dx - \lambda\norm{u_1-u_2}_{H^s(\R^n)}^2,\\
\norm{u_1}_{L^2(D)}& \leq c \norm{u_2}_{L^2(D)},
\end{align*}
for all $D\supseteq\supp(q_1-q_2)$, and all $F\in V_+^\perp\subseteq H_{q_1,q_2}(\Omega_e)$, where $c:=1+\frac{1}{\lambda}\norm{q_1-q_2}_{L^\infty(D)}$,
$u_1=S_{q_1}(F)$, and $u_2=S_{q_2}(F)$.

Combining Theorem \ref{Theorem for monotonicity} with its interchanged version, we obtain a subspace
\[
V\subseteq H_{q_1,q_2}(\Omega_e)\quad \text{ with } \quad \dim(V)\leq d(q_1)+d(q_2),
\]
so that 
\begin{align*}
\int_\Omega (q_1-q_2)|u_1|^2 \dx\leq
\Langle \left(\Lambda(q_1)-\Lambda(q_2)\right) F, F \Rangle 
\leq \int_\Omega (q_1-q_2)|u_2|^2 \dx.
\end{align*}
for all $F\in V^\perp\subseteq H_{q_1,q_2}(\Omega_e)$, $u_1=S_{q_1}(F)$, and $u_2=S_{q_2}(F)$.

Combining Theorem \ref{thm:useful} with its interchanged version, we obtain a subspace
\[
V_+\subseteq H_{q_1,q_2}(\Omega_e)\quad \text{ with } \quad \dim(V_+)\leq d(q_1)+d(q_2)+\dim(N_{q_1})+\dim(N_{q_2}),
\]
and constants $\lambda,c_1,c_2>0$, so that
\begin{align*}
\int_\Omega (q_1-q_2)|u_1|^2 \dx+\lambda\norm{u_1-u_2}_{H^s(\R^n)}^2 &\leq
\Langle \left(\Lambda(q_1)-\Lambda(q_2)\right) F, F \Rangle\\ 
&\leq \int_\Omega (q_1-q_2)|u_2|^2 \dx-\lambda\norm{u_1-u_2}_{H^s(\R^n)}^2,
\end{align*}
and
\[
c_1\norm{u_1}_{L^2(D)}\leq \norm{u_2}_{L^2(D)}\leq c_2 \norm{u_1}_{L^2(D)}
\]
for all $D\supseteq\supp(q_1-q_2)$, and all $F\in V_+^\perp\subseteq H_{q_1,q_2}(\Omega_e)$, 
$u_1=S_{q_1}(F)$, and $u_2=S_{q_2}(F)$.
\end{remark}

\begin{theorem}\label{thm:dimensions} 
	Let $d(q)$ be given by Definition \ref{def:dimension_d} and $N_q$ be defined by \eqref{N_q}.
\begin{enumerate}[(a)] 
\item For $q_1,q_2\in L^\infty(\Omega)$ 
\[
q_1\leq q_2 \quad \text{ implies } \quad d(q_1)\geq d(q_2).
\]
\item For all $q_1\in L^\infty(\Omega)$ there exists $\epsilon>0$ so that
\[
\dim(N_{q_1})\geq \dim(N_{q_2}) \quad \text{ for all } q_2\in L^\infty(\Omega)\text{ with } \norm{q_2-q_1}_{L^\infty(\Omega)}\leq \epsilon.
\]
\end{enumerate}
\end{theorem}

To prove Theorems \ref{Theorem for monotonicity}, \ref{thm:useful}, and \ref{thm:dimensions}, we first show the following lemmas.
\begin{lemma}\label{Lemma Mono 1}
	Let $q_1,q_2\in L^\infty(\Omega)$. Then, for all $F\in H_{q_1,q_2}(\Omega_e)$, 
	\begin{align*}
		\Langle \left(\Lambda(q_1)-\Lambda(q_2)\right) F, F \Rangle 
		+ \int_\Omega (q_2-q_1)|u_1|^2 \dx
		= \mathscr{B}_{q_2}(u_2-u_1,u_2-u_1),
	\end{align*}
	where $u_1=S_{q_1}(F)$, and $u_2=S_{q_2}(F)$. 
\end{lemma}
\begin{proof}
Using lemma \ref{lemma:NtD_and_bilinear_forms}, the assertion follows from
\begin{align*}
\lefteqn{\quad \mathscr{B}_{q_2}(u_2-u_1,u_2-u_1)= \mathscr{B}_{q_2}(u_2,u_2)-2 \mathscr{B}_{q_2}(u_2,u_1) + \mathscr{B}_{q_2}(u_1,u_1)}\\
&= - \mathscr{B}_{q_2}(u_2,u_1) + \mathscr{B}_{q_2}(u_1,u_1)
= - \mathscr{B}_{q_2}(u_2,u_1) + \mathscr{B}_{q_1}(u_1,u_1) + \int_\Omega (q_2-q_1)|u_1|^2 \dx\\
&= \Langle \left(\Lambda(q_1)-\Lambda(q_2)\right) F, F \Rangle+ \int_\Omega (q_2-q_1)|u_1|^2 \dx.
\end{align*}
\end{proof}

\begin{lemma}\label{Lemma Mono 2}
Let $q\in L^\infty(\Omega)$. Then there exists a subspace $W\subseteq H_0^s(\Omega)$ with $\dim(W)= d(q)$, and a constant $\lambda>0$, so
that 
\begin{alignat*}{2}
\mathscr{B}_{q}(w,w)&\geq 0 \quad && \text{ for all } w\in W^\perp\subseteq H_0^s(\Omega), \text{ and }\\
\mathscr{B}_{q}(w,w)&\geq \lambda\norm{w}_{H^s(\R^n)}^2 \quad && \text{ for all } w\in (W+N_q)^\perp\subseteq H_0^s(\Omega).
\end{alignat*}
\end{lemma}
\begin{proof}
Let $W$ be the sum of eigenspaces of the compact self-adjoint operator $\iota^*\iota - \iota^* M_q \iota$ corresponding to eigenvalues larger than $1$.
Then
\[
\mathscr{B}_{q}(w,w)=\left( (I-\iota^*\iota +\iota^* M_q \iota)w,w\right)_{H^s(\R^n)}\geq 0 \text{ for all } w\in W^\perp\subseteq H_0^s(\Omega).
\]
Since $N_q=\kernel(I-\iota^*\iota +\iota^* M_q \iota)$ is the eigenspace of $\iota^*\iota - \iota^* M_q \iota$ corresponding to the eigenvalue $1$, it also follows that
\[
\mathscr{B}_{q}(w,w)\geq (1-\mu)\norm{w}_{H^s(\R^n)}^2 \text{ for all } w\in (W+N_q)^\perp\subseteq H_0^s(\Omega),
\]
where $\mu$ is the largest eigenvalue of $\iota^*\iota - \iota^* M_q \iota$ smaller than $1$. Hence, the assertion follows with $\lambda:=1-\mu$.
\end{proof}

\begin{lemma}\label{Lemma Mono 3}
Let $q_1,q_2\in L^\infty(\Omega)$. There exists $\lambda>0$ and subspaces 
\[
V\subseteq V_+\subseteq H_{q_1,q_2}(\Omega_e)\quad \text{ with } \quad \dim(V)\leq d(q_2), \quad \dim(V_+)\leq d(q_2)+\dim(N_{q_2}),
\]
so that
\begin{alignat}{2}
\label{eq:mono3_semipos}\mathscr{B}_{q_2}(u_2-u_1,u_2-u_1)&\geq 0\quad && \text{ for all } F\in V^\perp\subseteq H_{q_1,q_2}(\Omega_e),\\
\label{eq:mono3_pos}\mathscr{B}_{q_2}(u_2-u_1,u_2-u_1)&\geq \lambda \norm{u_2-u_1}^2_{H^s(\R^n)} \quad &&\text{ for all } F\in V_+^\perp\subseteq H_{q_1,q_2}(\Omega_e),
\end{alignat}
where $u_1=S_{q_1}(F)$, and $u_2=S_{q_2}(F)$. 
\end{lemma}
\begin{proof}
The difference of the solution operators 
\[
S:\ H_{q_1,q_2}(\Omega_e)\to H_0^s(\Omega),\quad F\mapsto (S_{q_2}-S_{q_1})F=u_2-u_1\in H_0^s(\Omega),
\]
is linear and bounded by Corollary~\ref{cor:Dirichlet_bvp}. Using Lemma \ref{Lemma Mono 2} with $q:=q_2$ we obtain a subspace
$W\subseteq H_0^s(\Omega)$ with $\dim(W)= d(q_2)$, so that \eqref{eq:mono3_semipos} holds for all $F$ with $SF\in W^\perp$
which is equivalent to $F\in (S^*W)^\perp$. Also, by Lemma \ref{Lemma Mono 2}, \eqref{eq:mono3_pos} holds for all $F$ with $SF\in (W+N_q)^\perp$
which is equivalent to $F\in (S^*(W+N_q))^\perp$. Hence, the assertion follows with $V:=S^*W$, and $V_+:=S^*(W+N_q)$.
\end{proof}

\emph{Proof of Theorem~\ref{Theorem for monotonicity}.}
This immediately follows using the Lemmas \ref{Lemma Mono 1}--\ref{Lemma Mono 3}.
\endproof

\emph{Proof of Theorem~\ref{thm:useful}.}
The monotonicity relation \eqref{eq:useful_monotonicity} immediately follows using Lemmas \ref{Lemma Mono 1}--\ref{Lemma Mono 3}.
To prove \eqref{eq:useful_bound}, we use that
\[
0=\mathscr{B}_{q_1}(u_1,w)=\mathscr{B}_{q_2}(u_2,w) \quad \text{ for all } w\in H_0^s(\Omega),
\]
to conclude that for all $D\subseteq \Omega$ containing $\supp(q_1-q_2)$
\begin{align*}
\lefteqn{\quad \lambda \norm{u_2-u_1}^2_{H^s(\R^n)}\leq \mathscr{B}_{q_2}(u_2-u_1,u_2-u_1)=-\mathscr{B}_{q_2}(u_1,u_2-u_1)}\\
&=\mathscr{B}_{q_1}(u_1,u_2-u_1) -\mathscr{B}_{q_2}(u_1,u_2-u_1)= \int_\Omega (q_1-q_2) u_1 (u_2-u_1)\dx\\
&\leq \norm{q_1-q_2}_{L^\infty(D)}\norm{u_1}_{L^2(D)}\norm{u_2-u_1}_{H^s(\R^n)}.
\end{align*}
Hence
\[
\norm{u_2}_{L^2(D)}-\norm{u_1}_{L^2(D)}\leq \norm{u_2-u_1}_{L^2(D)}\leq \frac{1}{\lambda}\norm{q_1-q_2}_{L^\infty(D)}\norm{u_1}_{L^2(D)},
\]
which yields \eqref{eq:useful_bound} with $c:=1+\frac{1}{\lambda}\norm{q_1-q_2}_{L^\infty(D)}$.
\endproof

\emph{Proof of Theorem~\ref{thm:dimensions}.}
For $q_j\in L^\infty(\Omega)$, $j=1,2$, we denote  the positive eigenvalues (counted with multiplicities) of
the compact self-adjoint operator 
\[
\iota^* \iota - \iota^* M_{q_j} \iota:\ H_0^s(\Omega)\to H_0^s(\Omega),
\quad \text{by} \quad
\lambda_1^{(j)}\geq \lambda_2^{(j)}\geq \lambda_3^{(j)}\geq \ldots.
\]
\begin{enumerate}[(a)]
\item Let $q_1\leq q_2$. Then for all $v\in H_0^s(\Omega)$ 
\begin{align*}
\left( (\iota^* \iota - \iota^* M_{q_1} \iota) v,v\right)_{H_0^s(\Omega)}
&=\int_\Omega (1-q_1)|v|^2\dx\geq \int_\Omega (1-q_2)|v|^2\dx\\
&=\left( (\iota^* \iota - \iota^* M_{q_2} \iota) v,v\right)_{H_0^s(\Omega)}.
\end{align*}
Hence, it follows from the Courant-Fischer-Weyl min-max principle, (see, e.g., \cite{lax2002functional})
that 
\begin{align*}
\lambda_{k}^{(1)}&=
\max_{X\subset H_0^s(\Omega) \atop \dim(X)=k} \min_{\ v\in X \atop \norm{v}_{H_0^s(\Omega)}=1} \left( (\iota^* \iota - \iota^* M_{q_1} \iota) v,v\right)_{H_0^s(\Omega)}\\
&\geq \max_{X\subset H_0^s(\Omega) \atop \dim(X)=k} \min_{\ v\in X \atop \norm{v}_{H_0^s(\Omega)}=1} \left( (\iota^* \iota - \iota^* M_{q_2} \iota) v,v\right)_{H_0^s(\Omega)}=\lambda_{k}^{(2)},
\end{align*} 
for all $k\in \N$, which shows $d(q_1)\geq d(q_2)$.
\item Let $q_1\in L^\infty(\Omega)$. Since $N_{q_1}=\kernel(I-\iota^* \iota + \iota^* M_{q_1} \iota)$, 
exactly $\dim(N_{q_1})$ eigenvalues of $\iota^* \iota - \iota^* M_{q_1} \iota$ are identically one, so that
\[
\ldots\geq\lambda_{d(q_1)}^{(1)}> 1=\lambda_{d(q_1)+1}^{(1)}=\ldots=\lambda^{(1)}_{d(q_1)+\dim(N_{q_1})}>\lambda^{(1)}_{d(q_1)+\dim(N_{q_1})+1}\geq \ldots.
\]

Since $\lambda_{d(q_1)}^{(1)}-1>0$ and $ 1-\lambda^{(1)}_{d(q_1)+\dim(N_{q_1})+1}>0$, we can set
\begin{equation}\label{eq:EV_def_epsilon}
\epsilon:=\frac{1}{2}\min\left\{ \lambda_{d(q_1)}^{(1)}-1,\ 1-\lambda^{(1)}_{d(q_1)+\dim(N_{q_1})+1}\right\}>0.
\end{equation}

Then for all $q_2\in L^\infty(\Omega)$ with $\norm{q_2-q_1}_{L^\infty(\Omega)}\leq \epsilon$, and all $v\in H_0^s(\Omega)$
with $\norm{v}_{H_0^s(\Omega)}=1$,
we have
that 
\begin{align*}
\left| \left( (\iota^* \iota - \iota^* M_{q_1} \iota) v,v\right)_{H_0^s(\Omega)}
- \left( (\iota^* \iota - \iota^* M_{q_2} \iota) v,v\right)_{H_0^s(\Omega)} 
\right| \leq \int_{\Omega} |q_1-q_2| |v|^2\dx\leq \epsilon.
\end{align*}
Hence, using the Courant-Fischer-Weyl min-max principle as in (a) again, we obtain that
$\left|\lambda_{k}^{(1)}-\lambda_{k}^{(2)}\right|\leq \epsilon$ for all $k\in \N$.
In particular, using the definition of $\epsilon$ in \eqref{eq:EV_def_epsilon},
$\left|\lambda_{d(q_1)}^{(1)}-\lambda_{d(q_1)}^{(2)}\right|\leq \epsilon$ yields that
\[
\lambda_{1}^{(2)}\geq\ldots\geq\lambda_{d(q_1)}^{(2)}\geq \lambda_{d(q_1)}^{(1)}-\epsilon>1,
\]
and
$\left| \lambda_{d(q_1)+\dim(N_{q_1})+1}^{(1)}-\lambda_{d(q_1)+\dim(N_{q_1})+1}^{(2)}\right|\leq \epsilon$ yields that
\[
1>\lambda^{(1)}_{d(q_1)+\dim(N_{q_1})+1}+\epsilon \geq  \lambda^{(2)}_{d(q_1)+\dim(N_{q_1})+1}\geq \lambda^{(2)}_{d(q_1)+\dim(N_{q_1})+2}\geq \ldots .
\]
It follows that only the eigenvalues $\lambda^{(2)}_{d(q_1)+1},\ldots,\lambda^{(2)}_{d(q_1)+\dim(N_{q_1})}$ 
of $\iota^* \iota - \iota^* M_{q_2} \iota$
could possibly be identically one, so that $\dim(N_{q_2})\leq \dim(N_{q_1})$ is proven.
\end{enumerate}
\endproof

\subsection{Localized potentials for the fractional Schr\"{o}dinger equation}

In this subsection, we extend the localized potentials result that was derived in \cite{harrach2019nonlocal-mono1}
for positive potentials to general $L^\infty(\Omega)$-potentials and spaces of finite codimension.
Moreover, we will show a new result on controlling two localized potentials simultaneously.
We will prove the following two theorems.

\begin{theorem}[Localized potentials]\label{thm:localized}
Let $q\in L^\infty(\Omega)$. For every measurable set $M\subseteq \Omega$ with positive measure, and every finite-dimensional subspace
$V\subseteq H_q(\Omega_e)$ there exists a sequence $\{F^k\}_{k\in \N}\subseteq V^\perp\subseteq H_q(\Omega_e)$ so that
the corresponding solutions $u^k\in H_q^s(\R^n)$ of
\begin{equation}\label{fractional}
(-\Delta)^s u + qu = 0 \quad \text{ in $\Omega$,} \quad \text{ with } u|_{\Omega_e}=F^k,
\end{equation}
fulfill
\[
\int_{M} |u^k|^2 \dx \to \infty, \quad \text{ and } \quad \int_{\Omega\setminus M} |u^k|^2 \dx \to 0.
\]
\end{theorem}


\begin{theorem}[Simultaneously localized potentials]\label{thm:locpot2}
Let $q_1,q_2\in L^\infty(\Omega)$, and let $\supp(q_1-q_2)\subseteq M$ where $M\subseteq \Omega$ is a measurable set with positive measure.
For every finite-dimensional subspace $V\subseteq H_{q_1,q_2}(\Omega_e)$, there exists a sequence $\{F^k\}_{k\in \N}\subseteq V^\perp\subseteq H_{q_1,q_2}(\Omega_e)$ so that the corresponding solutions $u_1^k\in H^s_{q_1}(\R^n)$, $u_2^k\in H_{q_2}^s(\R^n)$, of
\begin{align*}\label{fractional_locpot2}
(-\Delta)^s u_1^k + q_1 u_1^k &= 0 \quad \text{ in $\Omega$,} \quad \text{ with } u_1^k|_{\Omega_e}=F^k,\\
(-\Delta)^s u_2^k + q_2 u_2^k &= 0 \quad \text{ in $\Omega$,} \quad \text{ with } u_2^k|_{\Omega_e}=F^k,
\end{align*}
fulfill
\begin{align*}
\int_{M} |u_1^k|^2 \dx &\to \infty, \quad  \quad \int_{\Omega\setminus M} |u_1^k|^2 \dx \to 0,\\
\int_{M} |u_2^k|^2 \dx &\to \infty, \quad  \quad \int_{\Omega\setminus M} |u_2^k|^2 \dx \to 0.
\end{align*}
\end{theorem}

To prove Theorem~\ref{thm:localized} and \ref{thm:locpot2}, we follow the general line of reasoning developed by one of the authors in \cite{gebauer2008localized}. We formulate the energy terms as norms of operator evaluations and characterize their adjoints and the ranges of their adjoints using the unique continuation property in Section \ref{subsect:UCP}. We then prove the two theorems using a functional analytic relation between norms of operator evaluations and ranges of their adjoints. 

We start by defining the so-called virtual measurement operators.
\begin{lemma}\label{lemma:Ladjoint}
For $q\in L^\infty(\Omega)$, a measurable set $M\subseteq \Omega$ with positive measure, and 
a subspace $H\subseteq H_q(\Omega_e)$ with finite codimension, we define the operator
\begin{align*}
L_{M,q}:&\ H\to L^2(M),\quad F\mapsto u|_M,
\end{align*}
where $u\in H_q^s(\R^n)$ solves 
\begin{equation}\label{eq:Ladjoint_frac}
(-\Delta)^s u + qu = 0 \quad \text{ in $\Omega$,} \quad \text{ with } u|_{\Omega_e}=F.
\end{equation}
Furthermore, let $V_M:=\{ u|_M:\ u\in N_q\}$. 

Then $L_{M,q}$ is a linear bounded operator, $\dim(V_M)<\infty$, and for all $g\in V_M^\perp\subseteq L^2(M)$
and $F\in H$
\begin{align}\label{equ: Ladjoint_frac}
\left( L_{M,q}^* g, F\right)_{H(\Omega_e)} 
=- \langle (-\Delta)^{s}v|_{\Omega_{e}},F \rangle,
\end{align}
where $v\in H_q^s(\R^n)$ solves $(-\Delta)^s v + q v=g\chi_M$ in $\Omega$, and $v|_{\Omega_e}=0$.
\end{lemma}
\begin{proof}
By Lemma \ref{lemma:solvability_Dirichlet_bvp} and Corollary \ref{cor:Dirichlet_bvp}, we have that $L_{M,q}$ is a linear bounded operator, $\dim(V_M)<\infty$, and for all $g\in V_M^\perp\subseteq L^2(M)$ there exists a solution $v\in H_q^s(\Omega)$ of $(-\Delta)^s v + q v=g\chi_M$ in $\Omega$, and $v|_{\Omega_e}=0$. Then $v\in H_0^s(\Omega)$ fulfills
\[
\mathscr B_q(v,w)=\int_M g w\dx \quad \text{ for all } w\in H_0^s(\Omega).
\]

For $F\in H$ let $u=u^{(0)}+u^{(F)}$ solve \eqref{eq:Ladjoint_frac} as in Lemma \ref{lemma:solvability_Dirichlet_bvp}. Then
\begin{align*}
\left( L_{M,q}^* g, F\right)_{H(\Omega_e)} &=\int_M g (L_{M,q}F) \dx = \int_M g u \dx
= \int_M g (u^{(0)}+u^{(F)}) \dx\\
&= \mathscr B_q(v,u^{(0)}) + \int_M g u^{(F)} \dx
= -\mathscr{B}_{q}(v,u^{(F)})+ \int_M g u^{(F)} \dx\\
&=- \int_{\R^n} (-\Delta)^{s/2} v \cdot (-\Delta)^{s/2} u^{(F)} \dx
+ \int_\Omega (-\Delta)^{s} v \cdot v^{(F)} \dx\\ 
&= -\langle (-\Delta)^{s}v|_{\Omega_{e}}, F\rangle.
\end{align*}
\end{proof}

We now proceed similarly to \cite{harrach2018helmholtz} to extend the functional analytic relation between the norms of two operators and the ranges of their adjoints from \cite[Lemma 2.5]{gebauer2008localized}, \cite[Corollary 3.5]{fruhauf2007detecting} to spaces of finite codimension.

\begin{lemma}\label{lemma:functional_analysis}
Let $X$, $Y$ and $Z$ be Hilbert spaces, $A_1:\ X\to Y$ and $A_2:\ X\to Z$ be linear bounded operators,
and let $N\subseteq X$ be a finite dimensional subspace. Then 
\[
\range(A_1^*)\subseteq \range(A_2^*)+N \quad \text{ if and only if } \quad \exists c>0:\ \norm{A_1x}\leq c\norm{A_2 x} \quad \forall x\in N^\perp,
\]
where $\range(A)$ denotes the range of the linear bounded operator $A$.
\end{lemma}
\begin{proof}
For both implications, we use that there exists an orthogonal projection operator $P_N:\ X\to X$ with
\[
\range(P_N)=N,\quad \kernel(P_N)=\range(I-P_N)=N^\perp, \quad \text{ and } \quad P_N^2=P_N=P_N^*.
\]

To show the first implication, let $\range(A_1^*)\subseteq \range(A_2^*)+N$.
Using block operator matrix notation we then have that
\[
\range(A_1^*)\subseteq \range(A_2^*)+\range(P_N)=\range\left( \begin{pmatrix}A_2^* & P_N\end{pmatrix} \right).
\]
Hence, by \cite[Lemma 2.5]{gebauer2008localized} there exists $c>0$ so that
\[
\norm{A_1x}^2\leq c^2\left\| \begin{pmatrix}A_2\\ P_N\end{pmatrix}x \right\|^2=c^2\norm{A_2x}^2+c^2\norm{P_N x}^2
\quad \text{ for all } x\in X,
\]
and thus
\[
\norm{A_1x}\leq c\norm{A_2x} \quad \text{ for all } x\in \kernel(P_N)=N^\perp.
\]

To show the converse implication, let $c>0$ and $\norm{A_1x}\leq c\norm{A_2 x}$ for all $x\in N^\perp$.
Then
\[
\norm{A_1(I-P_N)x}\leq c\norm{A_2(I-P_N)x} \quad \text{ for all } x\in X, 
\]
so that \cite[Lemma 2.5]{gebauer2008localized} yields that 
\[
\range ((I-P_N) A_1^*)\subseteq \range((I-P_N) A_2^*).
\]
Hence, 
\[
\range(A_1^*)\subseteq \range ((I-P_N) A_1^*) + N\subseteq \range((I-P_N) A_2^*) +N \subseteq \range(A_2^*)+N.
\]
\end{proof}

For the application of Lemma~\ref{lemma:functional_analysis}, the following elementary (and purely algebraic) observation will also be useful.
\begin{lemma}\label{lemma:algebraic}
Let $X$ and $Y$ be vector spaces, let $A:\ X\to Y$ be linear, and let $Y'$ be a subspace of $Y$. 
The following two statements are equivalent:
\begin{enumerate}[(a)]
\item There exists a finite dimensional subspace $N\subseteq Y$ with $A(X)\subseteq Y'+N$.
\item There exists a subspace $X'\subseteq X$ with finite codimension so that $A(X')\subseteq Y'$.
\end{enumerate}
Moreover, for all subspaces $X'\subseteq X$ with finite codimension, there exists a
finite dimensional subspace $N\subseteq Y$ with $A(X)\subseteq A(X')+N$, and $\dim(A(X'))=\infty$
holds if $\dim A(X)=\infty$.
\end{lemma}
\begin{proof}
Let $A(X)\subseteq Y'+N$, where $Y'$ and $N$ are subspaces of $Y$, and $\dim(N)<\infty$.
Since any basis of $N$ can be extended to a Hamel basis of $Y'+N$,
there exists a linear projection 
\[
P:\ Y'+N\to N\quad \text{ with }\quad \range(P)=N, \quad \text{ and } \quad \kernel(P)\subseteq Y'.
\]
Define $X':=\{ x\in X:\ PAx=0\}=\kernel(PA)$. Then 
\[
\mathrm{codim}(X')=\dim(\range(PA))\leq \dim(\range(P))=\dim(N),
\]
and by definition $A(X')\subseteq \kernel(P)\subseteq Y'$. This shows that (a) implies (b).

Clearly, (b) implies (a) by setting $N:=A(X'')$ where $X''$ is a linear complement of $X'$ in $X$.

Moreover, if $X'$ is a subspace of finite codimension then (b) holds with $Y'=A(X')$, so that (a)
implies the existence of a finite dimensional subspace $N\subseteq Y$ with $A(X)\subseteq A(X')+N$.
Clearly, this also implies that $\dim(A(X'))=\infty$ if $\dim(A(X))=\infty$.
\end{proof}

Now, we are ready to prove Theorem~\ref{thm:localized} and Theorem~\ref{thm:locpot2}.

\emph{Proof of Theorem~\ref{thm:localized}.}
Let $q\in L^\infty(\Omega)$, $M\subseteq \Omega$ be a measurable set with positive measure, and $V\subseteq H_q(\Omega_e)$ be a finite-dimensional subspace. As in Lemma~\ref{lemma:Ladjoint}, we define
the virtual measurement operators
\begin{alignat*}{2}
L_{M,q}:&\ H_q(\Omega_e)\to L^2(M),\quad && F\mapsto u|_M, \quad \text{ and }\\
L_{\Omega\setminus M,q}:&\ H_q(\Omega_e)\to L^2(\Omega \setminus M),\quad && F\mapsto u|_{\Omega \setminus M},
\end{alignat*}
where $u\in H_q^s(\R^n)$ solves 
\begin{equation*}
(-\Delta)^s u + qu = 0 \quad \text{ in $\Omega$} \quad \text{ with } u|_{\Omega_e}=F.
\end{equation*}

Then the assertion follows if we can show that there exists a sequence $\{F^k\}_{k\in \N}\subseteq V^\perp\subseteq H_q(\Omega_e)$ so that
\[
\norm{L_{M,q} F^k}_{L^2(M)}  \to \infty, \quad \text{ and } \quad \norm{L_{\Omega\setminus M,q} F^k}_{L^2(\Omega\setminus M)} \to 0.
\]
By a simple normalization argument (cf., e.g., the proof of \cite[Corollary 3.5]{harrach2019nonlocal-mono1}), it suffices to show that 
\[
\not\exists c>0:\ \norm{L_{M,q} F}_{L^2(M)}\leq c\norm{L_{\Omega\setminus M,q} F}_{L^2(\Omega\setminus M)}\quad \text{ for all } F\in V^\perp\subseteq H_q(\Omega_e).
\]
This follows from Lemma~\ref{lemma:functional_analysis} if we can show that
\begin{equation}\label{eq:locpot1_toshow}
\range(L_{M,q}^*)\not\subseteq \range(L_{\Omega\setminus M,q}^*)+V.
\end{equation}
We prove this by contradiction and assume that $\range(L_{M,q}^*)\subseteq \range(_{\Omega\setminus M,q}^*)+V$.

As in Lemma~\ref{lemma:Ladjoint}, define 
\[
V_M:=\{ u|_M:\ u\in N_q\}, \quad \text{ and } \quad V_{\Omega\setminus M}:=\{ u|_{\Omega\setminus M}:\ u\in N_q\}.
\]
Then $V_M^\perp$ and $V_{\Omega\setminus M}^\perp$ have finite codimension in $L^2(M)$ and $L^2(\Omega\setminus M)$, respectively.
Moreover, we define their subspaces
\begin{align*}
W_M&:=\left\{g\in V_M^\perp:\ \langle (-\Delta)^s v_M|_{\Omega_{e}},F\rangle=0\text{ for all } F\in H_q(\Omega_e)^\perp \right\},\\
W_{\Omega\setminus M}&:=\left\{g\in V_{\Omega\setminus M}^\perp:\ \langle (-\Delta)^s v_{\Omega\setminus M}|_{\Omega_{e}},F\rangle=0\text{ for all } F\in H_q(\Omega_e)^\perp \right\},
\end{align*}
where $v_M,v_{\Omega\setminus M}\in H_q^s(\R^n)$ are the solutions of 
\begin{alignat}{3}
\label{eq:locpot_proof_vM}(-\Delta)^s v_M + q v_M&=g_M\chi_M & \quad \text{ in $\Omega$}, && \qquad v_M|_{\Omega_e}&=0,\\
\label{eq:locpot_proof_vOM} (-\Delta)^s v_{\Omega\setminus M} + q v_{\Omega\setminus M}&=g_{\Omega\setminus M}\chi_{\Omega\setminus M}
  & \quad \text{ in $\Omega$}, && \qquad v_{\Omega\setminus M}|_{\Omega_e}&=0.
\end{alignat}
Then also $W_M$ and $W_{\Omega\setminus M}$ are subspaces of $L^2(M)$, resp., $L^2(\Omega\setminus M)$, with finite codimension, since the conditions in their definitions are equivalent to a system of finitely many homogeneous linear equations.

From Lemma~\ref{lemma:algebraic} we then obtain that
\[
L_{M,q}^*(W_M)\subseteq \range(L_{M,q}^*)\subseteq \range(L_{\Omega\setminus M,q}^*)+V
\subseteq L_{\Omega\setminus M,q}^*(W_{\Omega\setminus M})+V',
\]
with a finite-dimensional space $V'$. Moreover, using Lemma~\ref{lemma:algebraic} again,
there exists a subspace $W_M'\subseteq W_M$ with finite codimension in $W_M$ and thus in $L^2(M)$, so
that
\begin{equation}\label{eq:locpot_proof_range}
L_{M,q}^*(W_M')\subseteq L_{\Omega\setminus M,q}^*(W_{\Omega\setminus M}).
\end{equation}

Let $g_M\in W_M'$. Then, by \eqref{eq:locpot_proof_range}, 
there exists $g_{\Omega\setminus M}\in W_{\Omega\setminus M}$, so that the 
corresponding solutions $v_M,v_{\Omega\setminus M}\in H_q^s(\R^n)$ of \eqref{eq:locpot_proof_vM} and \eqref{eq:locpot_proof_vOM}
fulfill
\begin{equation*}
-\langle (-\Delta)^{s}v_M|_{\Omega_{e}}, F\rangle=\left( L_{M,q}^*g_M, F\right)_{H(\Omega_e)}=-\langle (-\Delta)^{s}v_{\Omega\setminus M}|_{\Omega_{e}}, F\rangle
\text{ for all } F\in H_q(\Omega_e),
\end{equation*}
where we have utilized \eqref{equ: Ladjoint_frac}.
By definition of $W_M$ and $W_{\Omega\setminus M}$, it also holds that
\[
\langle (-\Delta)^{s}v_M|_{\Omega_{e}}, F\rangle=0=\langle (-\Delta)^{s}v_{\Omega\setminus M}|_{\Omega_{e}}, F\rangle
\quad \text{ for all } F\in H_q(\Omega_e)^\perp.
\]
Hence $v:=v_M-v_{\Omega\setminus M}$ fulfills
\[
(-\Delta)^s v + q v=g_M\chi_M- g_{\Omega\setminus M}\chi_{\Omega\setminus M} \quad \text{ in $\Omega$}
\]
with vanishing Cauchy data $v|_{\Omega_e}=0$ and $(-\Delta)^{s}v|_{\Omega_{e}}=0$. From the unique continuation result in Corollary~\ref{cor:UCP}(b) it follows that $v\equiv 0$ in $\R^n$. But this yields $g_M=0$, and since this arguments holds for all $g_M\in W_M'$, it follows
that $W_M'=\{0\}$ which contradicts the fact that $W_M'$ is a subspace of finite codimension in the infinite dimensional space $L^2(M)$.
Hence, \eqref{eq:locpot1_toshow} and thus the assertion is proven.\quad \endproof

\emph{Proof of Theorem~\ref{thm:locpot2}.}
Let $q_1,q_2\in L^\infty(\Omega)$, and let $\supp(q_1-q_2)\subseteq M$ where $M\subseteq \Omega$ is a measurable set with positive measure. We first note that it suffices to
show that for all finite-dimensional subspaces $V\subseteq H_{q_1,q_2}(\Omega_e)$,
there exists a sequence $\{F^k\}_{k\in \N}\subseteq V^\perp\subseteq H_{q_1,q_2}(\Omega_e)$ with
\begin{align}\label{eq:locpot2_toshow1}
\int_{M} |u_1^k|^2 \dx &\to \infty, \quad \text{ and } \quad \int_{\Omega\setminus M} \left(|u_1^k|^2+|u_2^k|^2\right) \dx \to 0,
\end{align}
since $\int_{M} |u_1^k|^2 \dx \to \infty$ implies $\int_{M} |u_2^k|^2 \dx \to \infty$ on a subspace of 
finite codimension in $H_{q_1,q_2}(\Omega_e)$ by Remark \ref{remark_monotonicity}.

We define as in Lemma~\ref{lemma:Ladjoint},
\begin{alignat*}{2}
L_{M,q_1}:&\ H_{q_1,q_2}(\Omega_e)\to L^2(M),\quad && F\mapsto u_1|_M, \\
L_{\Omega\setminus M,q_1}:&\ H_{q_1,q_2}(\Omega_e)\to L^2(\Omega \setminus M),\quad && F\mapsto u_1|_{\Omega \setminus M}, \quad \text{ and }\\
L_{\Omega\setminus M,q_2}:&\ H_{q_1,q_2}(\Omega_e)\to L^2(\Omega \setminus M),\quad && F\mapsto u_2|_{\Omega \setminus M},
\end{alignat*}
where $u_j\in H_{q_j}^s(\R^n)$ solves (for $j=1,2$)
\[
(-\Delta)^s u_j + q_j u_j = 0 \quad \text{ in $\Omega$} \quad \text{ with } u_j|_{\Omega_e}=F.
\]
Thus \eqref{eq:locpot2_toshow1} can be reformulated as 
\[
\norm{L_{M,q_1}F^k}_{L^2(M)}\to \infty \quad \text{ and } \quad 
\left\|\begin{pmatrix}L_{\Omega\setminus M,q_1}\\ L_{\Omega\setminus M,q_2}\end{pmatrix} F^k\right\|_{L^2(\Omega\setminus M)\times L^2(\Omega\setminus M)}\to 0.
\]
Hence, using Lemma \ref{lemma:functional_analysis} as in the proof of Theorem~\ref{thm:localized}, the assertion follows if we can show that
\begin{equation}\label{eq:simlocpot_toshow}
\range(L_{M,q_1}^*)\not\subseteq \range\left(\begin{pmatrix}L_{\Omega\setminus M,q_1}^* & L_{\Omega\setminus M,q_2}^*\end{pmatrix}\right)+V
=\range(L_{\Omega\setminus M,q_1}^*)+\range(L_{\Omega\setminus M,q_2}^*)+V.
\end{equation}
We argue by contradiction and assume that
\begin{equation*}
\range(L_{M,q_1}^*)\subseteq \range(L_{\Omega\setminus M,q_1}^*)+\range(L_{\Omega\setminus M,q_2}^*)+V.
\end{equation*}

As in the proof of Theorem~\ref{thm:localized}, we define (for $j=1,2$)
\[
V_{M,q_1}:=\{ u|_M:\ u\in N_{q_1}\}, \quad  V_{\Omega\setminus M,q_j}:=\{ u|_{\Omega\setminus M}:\ u\in N_{q_j}\}.
\]
and 
\begin{align*}
W_{M,q_1}&:=\left\{g\in V_{M,q_1}^\perp:\ \langle (-\Delta)^s v_{M,q_1}|_{\Omega_{e}},F\rangle=0\text{ for all } F\in H_{q_1,q_2}(\Omega_e)^\perp \right\},\\
W_{\Omega\setminus M,q_j}&:=\left\{g\in V_{\Omega\setminus M,q_j}^\perp:\ \langle (-\Delta)^s v_{\Omega\setminus M,q_j}|_{\Omega_{e}},F\rangle=0\text{ for all } F\in H_{q_1,q_2}(\Omega_e)^\perp \right\},
\end{align*}
where $v_{M,q_1},v_{\Omega\setminus M,q_j}\in H_q^s(\R^n)$ are the solutions of 
\begin{alignat}{3}
\label{eq:locpot2_proof_vM}(-\Delta)^s v_{M,q_1} + q_1 v_{M,q_1}&=g_{M,q_1}\chi_M  \quad & \text{ in $\Omega$}, && \qquad v_{M,q_1}|_{\Omega_e}&=0,\\
\label{eq:locpot2_proof_vOM} (-\Delta)^s v_{\Omega\setminus M,q_j} + q_j v_{\Omega\setminus M,q_j}&=g_{\Omega\setminus M,q_j}\chi_{\Omega\setminus M}
 \quad & \text{ in $\Omega$}, && \qquad v_{\Omega\setminus M,q_j}|_{\Omega_e}&=0,
\end{alignat}
for $j=1,2$. Then, as in the proof of Theorem~\ref{thm:localized},
we obtain using lemma~\ref{lemma:algebraic} that 
\begin{equation}\label{eq:locpot2proof_subseteq}
L_{M,q_1}^*(W_{M,q_1}')\subseteq L_{\Omega\setminus M,q_1}^*(W_{\Omega\setminus M,q_1})+L_{\Omega\setminus M,q_2}^*(W_{\Omega\setminus M,q_2})
\end{equation}
with a subspace $W_{M,q_1}'\subseteq W_{M,q_1}$ that has finite codimension in $L^2(M)$.

Let $g_{M,q_1}\in W_{M,q_1}'$. As in the proof of Theorem~\ref{thm:localized}, it then follows from
\eqref{eq:locpot2proof_subseteq} and the definition of $W_{M,q_1}$, $W_{\Omega\setminus M,q_j}$, and $W_{\Omega\setminus M,q_2}$,
that there exist
$g_{q_j,\Omega\setminus M}\in W_{\Omega\setminus M,q_j}$ ($j=1,2$), so that
the solutions $v_{M,q_1}$, $v_{\Omega\setminus M,q_1}$, and $v_{\Omega\setminus M,q_2}$ of \eqref{eq:locpot2_proof_vM} and \eqref{eq:locpot2_proof_vOM} fulfill
\[
(-\Delta)^{s}v_{q_1,M}|_{\Omega_{e}}=(-\Delta)^{s}v_{q_1,\Omega\setminus M}|_{\Omega_{e}} + (-\Delta)^{s}v_{q_2,\Omega\setminus M}|_{\Omega_{e}}.
\]
It follows that $v:=v_{q_1,\Omega\setminus M}+v_{q_2,\Omega\setminus M}-v_{q_1,M}$ solves
\[
(-\Delta)^s v + q_1 v=g_{q_1,\Omega\setminus M}\chi_{\Omega\setminus M}+(q_1-q_2)v_{q_2,\Omega\setminus M} + g_{q_2,\Omega\setminus M}\chi_{\Omega\setminus M} -g_{q_1,M}\chi_M 
\]
with zero Cauchy data. Hence, by Corollary~\ref{cor:UCP}(b), $v=0$, and with $\supp(q_1-q_2)\subseteq M$ this also implies
\[
(q_1-q_2)v_{q_2,\Omega\setminus M}-g_{q_1,M}=0.
\]
Since $v_{q_2,\Omega\setminus M}\in H_0^s(\Omega)$, and the above arguments hold
for all $g_{M,q_1}\in W_{M,q_1}'$, it follows that
\[
W_{M,q_1}'\subseteq (M_{q_1}-M_{q_2})\iota (H_0^s(\Omega))\subseteq L^2(\Omega).
\]
Hence, the range of the compact operator $(M_{q_1}-M_{q_2})\iota$ would be a subspace of finite codimension 
in $L^2(\Omega)$ and thus closed. But the range of a compact operator can only be closed if it is finite dimensional 
(cf., e.g., \cite[Theorem.~4.18]{rudin1991functional}), so that this contradicts the infinite dimensionality of $L^2(\Omega)$.
Thus, \eqref{eq:simlocpot_toshow} is proven. \endproof

\begin{remark}
Our proof of the existence of simultaneously localized potentials followed the approach from \cite{gebauer2008localized}
that is based on a functional analytic relation between norms of operator evaluations and ranges of their adjoints. 
For some applications, cf., \cite{harrach2018helmholtz,harrach2018localizing}, and also in the first part of this work \cite{harrach2019nonlocal-mono1}, the existence of localized potentials also followed from Runge approximations arguments.
It is an interesting open question whether this alternative route of directly using Runge approximation could also yield 
an alternative proof of the existence of simultaneously localized potentials.
\end{remark}

\section{Converse monotonicity, uniqueness and inclusion detection}\label{Section 4}

Using the localized potentials and monotonicity relations from the last section, 
we can now extend the results from \cite{harrach2019nonlocal-mono1} to the case of a general potential $q\in L^\infty(\Omega)$.

\subsection{Converse monotonicity and the Calderón problem}

We first derive an if-and-only-if monotonicity relation between the potential and the DtN operators.
\begin{theorem}\label{Thm:if-and-only-if monotonicity}
Let $n\in \N$, $\Omega\subset \R^n$ be a Lipschitz bounded open set and $s\in (0,1)$. For any two potential $q_1,q_2 \in L^\infty(\Omega)$, we have 
	\begin{align}\label{converse mono relations}
		q_1 \geq q_2 \quad \text{ if and only if } \quad \Lambda(q_1) \geq_{d(q_2)} \Lambda(q_2)
		\quad \text{ if and only if } \quad \Lambda(q_1) \geq_{\text{fin}} \Lambda(q_2),
	\end{align}
	where $d(q_2)$ is the integer given in Section \ref{Section 3}.
\end{theorem}
\begin{proof}
Via Theorem \ref{Theorem for monotonicity}, $q_1 \geq q_2$ implies $\Lambda(q_1) \geq_{d(q_2)} \Lambda(q_2)$, and clearly
$\Lambda(q_1) \geq_{d(q_2)} \Lambda(q_2)$ implies $\Lambda(q_1) \geq_{\text{fin}} \Lambda(q_2)$. The assertion is proven if we can show that
$\Lambda(q_1) \geq_{\text{fin}} \Lambda(q_2)$ implies $q_1\geq q_2$ a.e. in $\Omega$.

Let $\Lambda(q_1) \geq_{\text{fin}} \Lambda(q_2)$. Using this together with Remark~\ref{remark_monotonicity} and 
that the intersection of subspaces with finite codimension still has finite codimension, we obtain a subspace $V\subseteq H_{q_1,q_2}(\Omega_e)$ so that 
\begin{align}\label{eq:iff_mono}
0 \leq \Langle \left(\Lambda(q_1)-\Lambda(q_2)\right) F, F \Rangle 
\leq \int_\Omega (q_1-q_2)|u_2|^2 \dx \quad \text{ for all } F\in V^\perp\subseteq H_{q_1,q_2}(\Omega_e),
\end{align}
where $u_2\in H_{q_2}^s(\R^n)$ solves 
\begin{equation}\label{eq:iff_u2}
(-\Delta)^s u_2 + q_2 u_2=0  \text{ in $\Omega$,} \quad \text{ and } \quad u_2|_{\Omega_e}=F.
\end{equation}
To show that this implies $q_1\geq q_2$ a.e. in $\Omega$, we argue by contradiction and assume
that there exists $\delta>0$ and a positive measurable set $M\subset\Omega$ such that $q_2-q_1\geq \delta$ on $M$. 
Then utilizing the localized potentials from Theorem \ref{thm:localized} we obtain a sequence $(F^k)_{k\in \N}\subset V^\perp\subseteq H_{q_1,q_2}(\Omega_e)$
where the corresponding solutions of \eqref{eq:iff_u2} with $F=F^k$ solve
\[
\int_{M} |u_2^k|^2 \dx \to \infty, \quad \text{ and } \quad \int_{\Omega\setminus M} |u_2^k|^2 \dx \to 0.
\]
But together with \eqref{eq:iff_mono} this yields to the contradiction
\begin{align*}
0 \leq \int_\Omega (q_1-q_2)|u_2^k|^2 \dx \leq -\delta \int_M |u_2^k|^2 \dx + \norm{q_1-q_2}_{L^\infty(\Omega)} \int_{\Omega\setminus M} |u_2^k|^2 \dx
\to -\infty,
\end{align*}
which proves $q_1\geq q_2$ a.e. in $\Omega$.
\end{proof}

\begin{corollary}\label{cor:Calderon}
Let $n\in \N$, $\Omega\subset \R^n$ be a bounded Lipschitz domain and $s\in (0,1)$.
For any two potentials $q_{0},q_{1}\in L_{+}^{\infty}(\Omega)$,
\begin{equation*}
q_{0}= q_{1}\mbox{ if and only if }\Lambda(q_0)\stackrel{\text{fin}}{=}\Lambda(q_1).
\end{equation*}
\end{corollary}
\begin{proof}
This follows immediately from Theorem~\ref{Thm:if-and-only-if monotonicity}.
\end{proof}

\subsection{A monotonicity-based reconstruction formula}

In \cite{harrach2019nonlocal-mono1}, we considered positive potentials $q\in L^\infty_+(\Omega)$, where $L^\infty_+(\Omega)$ denotes 
the set of all $L^\infty(\Omega)$-functions with positive essential infima.
We showed that $q\in L^\infty_+(\Omega)$ can be reconstructed from
$\Lambda(q)$ by taking the supremum of all positive density one simple functions $\psi$ with $\Lambda(\psi)\leq \Lambda(q)$.
The space of density one simple functions is defined by
\begin{align*}
\Sigma &:=\textstyle \left\{ \psi=\sum_{j=1}^m a_j \chi_{M_j}:\ a_j\in \mathbb{R},\ \text{$M_j\subseteq \Omega$ is a density one set} \right\},
\end{align*}
where we call a subset $M\subseteq \Omega$ a \emph{density one set} if it is non-empty, measurable and has Lebesgue density $1$ in 
all $x\in M$. Note that density one simple functions can be regarded as simple functions where function values that are only attained on a null set are replaced by
zero, and that, by the Lebesgue's density theorem, every measurable
set agrees almost everywhere with a density one set, so that every simple function agrees with a density one simple
function almost everywhere. For our results, it is important to control the values on null sets since these values 
might still affect the supremum when the supremum is taken over uncountably many functions.

For general potentials we obtain the following reconstruction formula.

\begin{theorem}\label{thm:constructive} Let $n\in \N$, $\Omega\subset \R^n$ be a bounded Lipschitz domain and $s\in (0,1)$. A potential $q\in L^{\infty}(\Omega)$
is uniquely determined by $\Lambda(q)$ via the following
formula
\begin{align*}
q(x)
&=\sup\{ \psi(x):\ \psi\in \Sigma,\ \Lambda(\psi)\leq_\text{fin} \Lambda(q) \}
+\inf\{ \psi(x):\ \psi\in \Sigma,\ \Lambda(\psi)\geq_\text{fin} \Lambda(q) \}\\
&=\sup\{ \psi(x):\ \psi\in \Sigma,\ \Lambda(\psi)\leq_{d(\psi)} \Lambda(q) \}
+\inf\{ \psi(x):\ \psi\in \Sigma,\ \Lambda(\psi)\geq_{d(q)} \Lambda(q) \}
\end{align*}
for a.e. $x\in \Omega$.
\end{theorem}

To prove Theorem~\ref{thm:constructive}, we first show the following lemma.
\begin{lemma}\label{lemma:sup_simple_functions}
For each function $q\in L^\infty(\Omega)$, and $x\in \Omega$ a.e., we have that
\begin{align*}
\max\{q(x),0\}&=\sup\{ \psi(x):\ \psi\in \Sigma \text{ with } \psi\leq q \}.
\end{align*}
\end{lemma}
\begin{proof}
Let $q\in L^\infty(\Omega)$. By the standard simple function approximation lemma, cf., e.g., \cite{royden1988real}, there exists a sequence
$(\psi_k),{k\in \N}$ of simple functions with 
\begin{equation}\label{eq:simple_fct_approx_aux}
q(x)-\frac{1}{k} \leq \psi_k(x)\leq q(x)
\end{equation}
for all $k\in \N$ and $x\in \Omega$. Since every simple function agrees with a density one simple function almost everywhere, we can change 
the values of the countably many functions $\psi_k$ on a null set, to obtain $\psi_k\in \Sigma$ for which
\eqref{eq:simple_fct_approx_aux} holds almost everywhere. Hence, for a.e. $x\in \Omega$,
\begin{align*}
q(x)=\lim_{k\to \infty} \psi_k (x) \leq \sup\{ \psi(x):\ \psi\in \Sigma,\ \psi\leq q\}.
\end{align*}
Moreover, if $x\in \Omega$ then $\psi_x=-\norm{q}_{L^\infty(\Omega)} \chi_{\Omega\setminus \{x\}}$ 
is a density one simple function fulfilling
$\psi_x(x)=0$ and $\psi_x(\xi)\leq q(\xi)$ for a.e. $\xi\in \Omega$, so that $\psi_x\leq q$. Hence,
\[
0\leq \sup\{ \psi(x):\ \psi\in \Sigma,\ \psi\leq q\} \quad \text{ for a.e. } x\in \Omega.
\]

It remains to show that
\begin{equation}\label{eq:sup_simple_geq}
\max\{q(x),0\}\geq \sup\{ \psi(x):\ \psi\in \Sigma,\ \psi\leq q\} \quad \text{ for a.e. }x\in \Omega.
\end{equation}
We argue as in the proof of \cite[Lemma~4.4]{harrach2019nonlocal-mono1}. It suffices to show that for each $\delta>0$ the set 
\begin{equation}\label{eq:sup_simple_M}
M:=\{ x\in \Omega:\ \max\{q(x),0\}+\delta< \sup\{ \psi(x):\ \psi\in \Sigma,\ \psi\leq q\} \}
\end{equation}
is a null set. To prove this, assume that $M$ is not a null set for some $\delta>0$. 
By removing a null set from $M$, we can assume that $M$ is a density one set.
By using Lusin's theorem (see \cite{royden1988real} for instance), all measurable function are approximately continuous at almost every point. Hence, $M$ must contain a point
$\widehat x$ in which the function $x\mapsto \max\{q(x),0\}$ is approximately continuous, and thus the set
\[
M':=\{ x\in \Omega: \max\{q(x),0\}\leq \max\{q(\widehat x),0\}+\delta / 3 \}
\]
has density one in $\widehat x$.  Removing a null set, we can assume that $M'$ is a density one set still containing $\widehat x$.

Moreover, by the definition of $M$, there must exist a $\psi\in \Sigma$ with $\psi\leq q$ and
\[
\max\{q(\widehat x),0\}+\frac{2}{3}\delta \leq \psi(\widehat x).
\]
This shows $\psi(\widehat x)>0$, so that, by \cite[Lemma~4.3]{harrach2019nonlocal-mono1},
there exists a density one set $M''$ containing $\widehat x$, where $\psi(x)=\psi(\widehat x)$ for all $x\in M''$.

We thus have that for all $x\in M'\cap M''$
\[
q(x)+\delta/3\leq \max\{q(x),0\}+\delta/3 \leq \max\{q(\widehat x),0\}+\frac{2}{3}\delta \leq \psi(\widehat x) = \psi(x),
\]
and $M'\cap M''$ possesses positive measure since $M'$ and $M''$ are density one sets that both contain $\widehat x$, cf., again, \cite[Lemma~4.3]{harrach2019nonlocal-mono1}.
But this contradicts that $\psi(x)\leq q(x)$ almost everywhere, and thus shows that $M$ defined in \eqref{eq:sup_simple_M} is a null set for all $\delta>0$. It follows that
\eqref{eq:sup_simple_geq} holds, so that the assertion is proven.
\end{proof}

\emph{Proof of Theorem \ref {thm:constructive}.}
Using lemma~\ref{lemma:sup_simple_functions} and the if-and-only-if monotonicity relation in Theorem \ref{Thm:if-and-only-if monotonicity},
we have that for all $q\in L^\infty(\Omega)$, and all $x\in \Omega$ a.e.,
\begin{align*}
q(x)&=\max\{q(x),0\}-\max\{ -q(x),0\}\\
&=\sup\{ \psi(x):\ \psi\in \Sigma,\ \psi\leq q \}
-\sup\{ \psi(x):\ \psi\in \Sigma,\ \psi\leq -q \}\\
&=\sup\{ \psi(x):\ \psi\in \Sigma,\ \psi\leq q \}
+\inf\{ \psi(x):\ \psi\in \Sigma,\ \psi\geq q \}\\
&=\sup\{ \psi(x):\ \psi\in \Sigma,\ \Lambda(\psi)\leq_\text{fin} \Lambda(q) \}
+\inf\{ \psi(x):\ \psi\in \Sigma,\ \Lambda(\psi)\geq_\text{fin} \Lambda(q) \}\\
&=\sup\{ \psi(x):\ \psi\in \Sigma,\ \Lambda(\psi)\leq_{d(\psi)} \Lambda(q) \}
+\inf\{ \psi(x):\ \psi\in \Sigma,\ \Lambda(\psi)\geq_{d(q)} \Lambda(q) \}.
\end{align*}
This completes the proof.
\endproof

\subsection{The linearized Calderón problem}

In this subsection, we will only consider $q\in L^\infty(\Omega)$ that fulfill the following assumption.

\begin{definition}
Let $N_q$ be the set defined by \eqref{N_q}, then we say that $q\in L^\infty(\Omega)$ is \emph{non-resonant}, if $N_q=\{0 \}$.
\end{definition}

This assumption is also called an \emph{eigenvalue condition} in the literature, since it is equivalent to $\{0\}$ being not an Dirichlet eigenvalue of the fractional operator $(-\Delta)^s +q$ in $\Omega$.
Note that it implies that $H_q(\Omega_{e})=H(\Omega_{e})$, and $H_q^s(\mathbb R^n)=H^s(\mathbb R^n)$,
i.e., that the Dirichlet problem is uniquely solvable for all Dirichlet data in $H(\Omega_{e})$, cf.\ Corollary \ref{cor:Dirichlet_bvp}.

We start by showing that the non-resonant potentials are an open subset of $L^\infty(\Omega)$, on
which the DtN operator is Fr\'echet differentiable.

\begin{lemma}\label{lemma:Frechet_derivative}
The set $\mathcal O=\{ q\in L^\infty(\Omega):\ N_q=\{ 0 \} \}$ is an open subset of $L^\infty(\Omega)$. On this set, the DtN operator
\[
\Lambda:\ \mathcal O\subseteq L^\infty(\Omega)\to \mathcal{L}(H(\Omega_{e}),H(\Omega_{e})^*),
\quad q\mapsto \Lambda(q),
\]
is Fr\'echet differentiable. For each $q\in \mathcal O$ its derivative is given by
\begin{align*}
\Lambda'(q):& \ L^\infty(\Omega) \to \mathcal{L}(H(\Omega_{e}),H(\Omega_{e})^*),\quad r\mapsto \Lambda'(q)r,\\
\left\langle (\Lambda'(q) r) F, G \right\rangle:&=\int_{\Omega} r S_q(F) S_q(G) \dx \quad \text{ for all } r\in L^\infty(\Omega),\ F,G\in H(\Omega_{e}),
\end{align*}
where $S_q:\ H(\Omega_{e})\to H^s(\mathbb R^n)$, $F\mapsto u$,
is the solution operator of the Dirichlet problem
\[
(-\Delta)^{s}u + q u=0  \mbox{ in }\Omega \quad\text{ and } \quad u|_{\Omega_{e}}=F.
\]
\end{lemma}
\begin{proof}
The fact that $\mathcal O$ is open immediately follows from Theorem \ref{thm:dimensions}(b).

Let $q\in \mathcal O\subseteq L^\infty_+(\Omega)$. $\Lambda'(q)$ is a linear bounded operator since $S_q$ is linear and bounded, cf.\ Corollary \ref{cor:Dirichlet_bvp}. For sufficiently small $r\in L^\infty(\Omega)$, we have that $q+r\in \mathcal O$, and it 
follows from Lemma~\ref{lemma:NtD_and_bilinear_forms} that 
\[
\left\langle \left( \Lambda(q+r) - \Lambda(q)\right) F, F \right\rangle= \int_\Omega r S_{q+r}(F)\, S_{q}(F)\dx.
\]
With the operator formulation from the proof of Lemma~\ref{lemma:solvability_Dirichlet_bvp}, it is then easy to show that, for sufficiently small $r\in L^\infty(\Omega)$, there exists a constant $C>0$ with
\[
\norm{S_{q+r}(F)-S_q(F)}_{H^s(\R^n)}\leq C \norm{r}_{L^\infty(\Omega)} \norm{F}_{H(\Omega_e)}.
\]

Using that $\Lambda(q)$, $\Lambda(q+r)$, and $\Lambda'(q)r$ are symmetric operators, it now follows that
\begin{align*}
\lefteqn{ \| \Lambda(q+r)-\Lambda(q)-\Lambda'(q)r\|_{\mathcal{L}(H(\Omega_{e}),H(\Omega_{e})^*)}}\\[+1ex] 
&= \sup_{\| F\|_{H(\Omega_{e})} =1 } 
\left| \left\langle \left(\Lambda(q+r)-\Lambda(q)-\Lambda'(q)r\right) F,F\right\rangle \right|\\
&= \sup_{\| F\|_{H(\Omega_{e})} =1 } 
\left| \int_\Omega r (S_{q+r}(F)-S_q(F)) S_{q}(F)\dx \right|
\leq  C\norm{r}_{L^\infty(\Omega)}^2 \norm{S_{q}}_{\mathcal{L}(H(\Omega_{e}),H^s(\R^n))}.
\end{align*}
which proves the assertion.
\end{proof}

Using the Fr\'echet derivative from Lemma~\ref{lemma:Frechet_derivative}, the monotonicity relations in Theorem \ref{Theorem for monotonicity} and \ref{thm:useful} can now be written as follows.
\begin{corollary}\label{remark:monotonicity_with_Frechet}
For all non-resonant $q_{1},q_{2}\in L^{\infty}(\Omega)$,
\begin{align*}
\Lambda'(q_2)(q_1-q_2) \geq_{d(q_1)}
\Lambda(q_1)-\Lambda(q_2)
\geq_{d(q_2)} \Lambda'(q_1)(q_1-q_2),
\end{align*}
and there exists $c>0$ so that for all measurable $D\subseteq \Omega$ containing $\supp(q_1-q_2)$
\[
c \Lambda'(q_2)\chi_D  \geq_{d(q_1)}  \Lambda'(q_1)\chi_D  \geq_{d(q_2)} \frac{1}{c} \Lambda'(q_2)\chi_D.
\]
\end{corollary}
\begin{proof}
Since $q_{1},q_{2}\in L^{\infty}(\Omega)$ are non-resonant, we have that $H_{q_1}(\Omega_e)=H_{q_2}(\Omega_e)=H_{q_1,q_2}(\Omega_e)=H(\Omega_e)$.
It then follows from Theorem \ref{Theorem for monotonicity} and Lemma~\ref{lemma:Frechet_derivative} that there exists a subspace $V\subseteq H(\Omega_e)$ with $\dim(V)\leq d(q_2)$
so that for all $F\in V^\perp$
\[
\Langle \left(\Lambda(q_1)-\Lambda(q_2)\right)F,F \Rangle \geq \int_\Omega (q_1-q_2) S_{q_1}(F)^2\dx
= \Langle \left( \Lambda'(q_1)(q_1-q_2) \right) F,F \Rangle, 
\]
which shows that $\Lambda(q_1)-\Lambda(q_2)\geq_{d(q_2)} \Lambda'(q_1)(q_1-q_2)$.

Also, it follows from Theorem \ref{thm:useful} and Lemma \ref{lemma:Frechet_derivative} that there exists 
a subspace $V_+\subseteq H(\Omega_e)$ with $\dim(V_+)\leq d(q_2)+N(q_2)=d(q_2)$ and a constant $\widetilde c>0$, so that
for all measurable $D\subseteq \Omega$ containing $\supp(q_1-q_2)$, and all $F\in V_+^\perp$,
\[
\Langle \left( \Lambda'(q_2)\chi_D \right) F,F \Rangle
=\norm{S_{q_2}(F)}_{L^2(D)}^2\leq \widetilde c^2 \norm{S_{q_1}(F)}_{L^2(D)}^2
=\widetilde c^2  \Langle \left( \Lambda'(q_1)\chi_D \right) F,F \Rangle,
\]
which shows $\Lambda'(q_2)\chi_D \leq_{d(q_2)} c \Lambda'(q_1)\chi_D$ with $c:=\widetilde c^2$. 

The other assertions follow by interchanging $q_1$ and $q_2$.
\end{proof}

We also have an if-and-only if monotonicity result for the linearized DtN-operators. 
\begin{theorem}\label{thm:converse_mon_frechet}
Let $n\in \N$, $\Omega\subset \R^n$ be a Lipschitz bounded open set and $s\in (0,1)$. Then for all non-resonant $q\in L^\infty(\Omega)$ and $r_1,r_2\in L^\infty(\Omega)$,
\[
r_1\leq r_2 \quad \text{ if and only if } \quad \Lambda'(q)r_1\leq \Lambda'(q)r_2
\quad \text{ if and only if } \quad \Lambda'(q)r_1\leq_\text{fin} \Lambda'(q)r_2.
\]
\end{theorem}
\begin{proof}
If $r_1\leq r_2$ then $\Lambda'(q)r_1\leq \Lambda'(q)r_2$ follows immediately from the characterization 
of $\Lambda'(q)$ in Lemma \ref{lemma:Frechet_derivative}. (Note that this holds on the whole space $H(\Omega_e)$, and not just
on a subspace of finite codimension).

Clearly, $\Lambda'(q)r_1\leq \Lambda'(q)r_2$ implies
$\Lambda'(q)r_1\leq_\text{fin} \Lambda'(q)r_2$, and the implication from 
$\Lambda'(q)r_1\leq_\text{fin} \Lambda'(q)r_2$ to $r_1\leq r_2$ follows from the same localized potentials argument
as in the proof of Theorem \ref{Thm:if-and-only-if monotonicity}.
\end{proof}

This implies uniqueness of the linearized fractional Calder\'on problem:
\begin{corollary}\label{cor:Calderon_linearized}
Let $n\in \N$, $\Omega\subset \R^n$ be a Lipschitz bounded open set and $s\in (0,1)$. 
For all non-resonant $q\in L^\infty(\Omega)$, the Fr\'echet derivative $\Lambda'(q)$ is injective, i.e.
\begin{equation*}
\Lambda'(q) r\stackrel{\text{fin}}{=}0 \quad \text{ if and only if } \quad 
\Lambda'(q) r=0 \quad \text{ if and only if } \quad r=0.
\end{equation*}
\end{corollary}
\begin{proof}
This follows immediately from Theorem~\ref{thm:converse_mon_frechet}.
\end{proof}

\subsection{Inclusion detection by linearized monotonicity tests}\label{subsect:linearized_inclusion_detection}

In this section we will study the inclusion detection (or shape reconstruction) problem of determining regions where a non-resonant potential $q\in L^\infty(\Omega)$ changes from a known non-resonant reference potentials $q_0\in L^\infty(\Omega)$, i.e., we aim to reconstruct the support $q-q_0$ by comparing $\Lambda(q)$ with $\Lambda(q_0)$.
$q_0$ may describe a background coefficient, and $q$ denotes the coefficient function in the presence of anomalies or scatterers. 

We will generalize the results in \cite{harrach2019nonlocal-mono1} and show that the support of $q-q_0$ can be reconstructed with \emph{linearized monotonicity tests} \cite{harrach2013monotonicity,garde2017comparison}. These linearized tests only utilize the solution of the fractional Schr\"odinger equation
with the reference coefficient function $q_0\in L^\infty(\Omega)$. They do not require any other special solutions of the equation.

In all of the following let $n\in \N$, $\Omega\subset \R^n$ be a Lipschitz bounded open set, $s\in (0,1)$,
and let $q_0,q\in L^\infty(\Omega)$ be non-resonant.

For a measurable subset $M\subseteq \Omega$, we introduce the testing operator 
$\mathcal{T}_{M}:H(\Omega_{e})\to H(\Omega_{e})^{*}$ by 
setting $\mathcal{T}_{M}:=\Lambda'(q_0)\chi_M$. i.e.,
\begin{equation}
\label{TB operator}
\left\langle \mathcal{T}_{M}F,G\right\rangle :=\int_{M} S_{q_0}(F) S_{q_0}(G) dx
\quad \text{ for all } F,G\in H(\Omega_{e}),
\end{equation}
where $S_{q_0}:\ H(\Omega_{e})\to H^s(\mathbb R^n)$ denotes the solution operator
as in as in Lemma~\ref{lemma:Frechet_derivative}. 

The following theorem shows that we can find the support of $q-q_0$ by shrinking closed sets, cf.\ \cite{harrach2013monotonicity,garde2019regularized}.
\begin{theorem}\label{thm:support_from_closed_sets}
For each closed subset $C\subseteq \Omega$, 
\begin{align*}
\lefteqn{\mathrm{supp}(q-q_0)\subseteq C,}\\
& \quad \text{ if and only if } \quad  \exists \alpha>0:\ 
-\alpha \mathcal{T}_C \leq_{d(q_0)+d(q)} \Lambda(q)-\Lambda(q_0)\leq_{d(q)}  \alpha \mathcal{T}_C,\\
& \quad \text{ if and only if } \quad  \exists \alpha>0:\ 
-\alpha \mathcal{T}_C \leq_\text{fin} \Lambda(q)-\Lambda(q_0)\leq_\text{fin}  \alpha \mathcal{T}_C.
\end{align*}
Hence,
\begin{align*}
\lefteqn{\mathrm{supp}(q-q_0)}\\
&=\bigcap \{ C\subseteq \Omega \text{ closed}:\ \exists \alpha>0:\ -\alpha \mathcal{T}_C \leq_{d(q_0)+d(q)}  \Lambda(q)-\Lambda(q_0)\leq_{d(q)} \alpha \mathcal{T}_C\}\\
&=\bigcap \{ C\subseteq \Omega \text{ closed}:\ \exists \alpha>0:\ -\alpha \mathcal{T}_C \leq_\text{fin}  \Lambda(q)-\Lambda(q_0)\leq_\text{fin} \alpha \mathcal{T}_C\}.
\end{align*}
\end{theorem}
\begin{proof}
\begin{enumerate}
\item[(a)] Let $\mathrm{supp}(q-q_0)\subseteq C$. Then, by Corollary \ref{remark:monotonicity_with_Frechet},
there exists a constant $c>0$ with
\[
\mathcal{T}_C=\Lambda'(q_0)\chi_C\geq_{d(q)} c \Lambda'(q)\chi_C.
\]
Moreover, $\mathrm{supp}(q-q_0)\subseteq C$ implies that for sufficiently large $\alpha>0$
\[
- \alpha c \chi_C \leq q-q_0 \leq \alpha \chi_C.
\]
Using Corollary \ref{remark:monotonicity_with_Frechet} and Theorem \ref{thm:converse_mon_frechet}, we thus obtain
\begin{align*}
\Lambda(q)&\leq_{d(q)} \Lambda(q_0)+\Lambda'(q_0)(q-q_0)\leq \Lambda(q_0)+\Lambda'(q_0)\alpha \chi_C 
= \Lambda(q_0) + \alpha \mathcal{T}_C,\\
\Lambda(q)&\geq_{d(q_0)} \Lambda(q_0)+\Lambda'(q)(q-q_0)\geq \Lambda(q_0)-\alpha c \Lambda'(q)\chi_C
\geq_{d(q)} \Lambda(q_0) - \alpha \mathcal{T}_C.
\end{align*}
\item[(b)] We will now show that
\begin{equation}\label{eq:proof_TC_alpha}
\exists \alpha>0:\ -\alpha \mathcal{T}_C \leq_\text{fin} \Lambda(q)-\Lambda(q_0)\leq_\text{fin}  \alpha \mathcal{T}_C
\end{equation}
implies $\mathrm{supp}(q-q_0)\subseteq C$. 

Let $\alpha>0$ fulfill \eqref{eq:proof_TC_alpha}. Then we obtain from the first inequality in \eqref{eq:proof_TC_alpha} with Corollary \ref{remark:monotonicity_with_Frechet} 
\[
\Lambda'(q_0)(-\alpha\chi_C)=-\alpha \mathcal{T}_C \leq_\text{fin} \Lambda(q)-\Lambda(q_0)
\leq_\text{fin} \Lambda'(q_0)(q-q_0),
\]
so that Theorem \ref{thm:converse_mon_frechet} yields that 
\begin{equation}\label{eq:proof_TC_alphachileqqq0}
-\alpha\chi_C\leq q-q_0.
\end{equation}

It remains to show that the second inequality in \eqref{eq:proof_TC_alpha} implies that
\begin{equation}\label{eq:proof_TC_complicated_part}
q-q_0\leq 0 \quad \text{ on $\Omega\setminus C$.}
\end{equation}
We argue by contradiction and assume that \eqref{eq:proof_TC_complicated_part} is not true. Then there exists $\delta>0$, and a measurable subset $M\subseteq \Omega\setminus C$ with positive measure so that $q-q_0\geq \delta$ on $M$. 

We now use an idea from \cite{harrach2019global} to rewrite energy terms by repeated application of the monotonicity relation, and define
\[
\widetilde q:=q+\delta \chi_M - \alpha \chi_C + (q_0-q)\chi_{\Omega\setminus(M\cup C)}=\left\{\begin{array}{c l} q+\delta & \text{ in $M$,}\\
q-\alpha   & \text{ in $C$,}\\
q_0  & \text{ in $\Omega\setminus(M\cup C)$,}\\
\end{array}\right.
\]
and note that 
\[
q-q_0\geq \delta\chi_M - \alpha \chi_C\geq \delta\chi_M - \alpha \chi_C+(q_0-q)\chi_{\Omega\setminus (C\cup M)}=\widetilde q-q.
\]

Using Theorem~\ref{Theorem for monotonicity} and Remark~\ref{remark_monotonicity}, there exists a finite dimensional subspace $V\subseteq H_{\widetilde q}(\Omega_e)$
so that for all $F\in V^\perp\subseteq H_{\widetilde q}(\Omega_e)$
\begin{align}\nonumber
\Langle \left(\Lambda(q)-\Lambda(q_0)\right) F, F \Rangle 
&\geq \int_\Omega (q-q_0)|u_q|^2 \dx\geq \int_\Omega (\widetilde q-q)|u_q|^2 \dx\\
\label{eq:proof_TC_long_inequality}
&\geq \Langle \left(\Lambda(\widetilde q)-\Lambda(q)\right) F, F \Rangle
\geq \int_\Omega (\widetilde q-q)|u_{\widetilde q}|^2 \dx\\
\nonumber &\geq \delta \int_{M} |u_{\widetilde q}|^2 \dx - \alpha \int_{\Omega\setminus M} |u_{\widetilde q}|^2\dx,
\end{align}
where $u_q=S_q(F)$, $u_{\widetilde q}=S_{\widetilde q}(F)$, and, for the last inequality, we assumed without loss of generality that $\alpha>0$ is larger than $\norm{q-q_0}_{L^\infty(\Omega)}$. For the last argument, note
that the inequalities in \eqref{eq:proof_TC_long_inequality} each hold on possibly different subspaces of finite codimension
in $H_{\widetilde q}(\Omega_e)$, so that $V$ is obtained by taking the orthogonal complement of the intersection of all these spaces.

We also define 
\[
\widetilde q_0:=\left\{\begin{array}{c l} q-\alpha   & \text{ in $C$,}\\
q_0  & \text{ in $\Omega\setminus C$.}
\end{array}\right.
\]
Since $\supp(\widetilde q_0-q_0)\subseteq C$, we can apply Theorem~\ref{thm:useful} 
to obtain a finite dimensional subspace $V'\subseteq H_{\widetilde q_0,q_0}(\Omega_e)= H_{\widetilde q_0}(\Omega_e)$ (note that $q_0$ is non-resonant), and a constant $c>0$,
so that for all $F\in V'^\perp\subseteq H_{\widetilde q_0}(\Omega_e)$
\[
\langle \mathcal{T}_C F,F\rangle=
\int_{C} |u_{q_0}|^2\dx \leq c \int_{C} |u_{\widetilde q_0}|^2\dx,
\]
where $u_{q_0}=S_{q_0}(F)$, $u_{\widetilde q_0}=S_{\widetilde q_0}(F)$.
Hence, the second inequality in \eqref{eq:proof_TC_alpha} implies that 
\begin{equation}\label{eq:linproof_contradiction}
c\int_{C} |u_{\widetilde q_0}|^2\dx\geq \delta \int_M |u_{\widetilde q}|^2 \dx - \alpha \int_{\Omega\setminus M} |u_{\widetilde q}|^2 \dx
\end{equation}
for all $F\in W^\perp\subseteq H_{\widetilde q_0,\widetilde q}(\Omega_e)$, where $W\subseteq H_{\widetilde q_0,\widetilde q}(\Omega_e)$
is a finite dimensional subspace.
But $\supp(\widetilde q-\widetilde q_0)\subseteq M$, so that the result on simultaneously localized potentials in Theorem \ref{thm:locpot2}
(with Theorem \ref{thm:locpot2} applied to the herein constructed subspace $W$)
yields the existence of a sequence $\{F^k\}_{k\in \N}\subseteq W^\perp\subseteq H_{\widetilde q_0,\widetilde q}(\Omega_e)$,
so that the corresponding solutions $u_{\widetilde q_0}^k=S_{\widetilde q_0}(F^k)$, $u_{\widetilde q}^k=S_{\widetilde q}(F^k)$,
fulfill
\begin{align*}
\int_{M} |u_{\widetilde q}^k|^2 \dx \to \infty, \quad  \int_{\Omega\setminus M} |u_{\widetilde q_0}^k|^2 \dx \to 0, 
 \quad \text{ and } \quad \int_{\Omega\setminus M} |u_{\widetilde q}^k|^2 \dx \to 0,
\end{align*}
which contradicts \eqref{eq:linproof_contradiction} since $C\subseteq \Omega\setminus M$.
Hence, \eqref{eq:proof_TC_complicated_part} and thus the assertion is proven.
\end{enumerate}
\end{proof}

We also extend the simpler results for the definite case, where either $q\geq q_0$ or $q\leq q_0$
holds almost everywhere in $\Omega$, from \cite{harrach2019nonlocal-mono1} to general (but non-resonant) $L^\infty(\Omega)$-potentials. 
We will show that it suffices to test open balls to reconstruct 
the inner support (for $q\geq q_0$), resp., a set between the support of $q-q_0$
and its inner support (for $q\leq q_0$), where, as in \cite[Section 2.2]{harrach2013monotonicity}, the inner support $\mathrm{inn\,supp}(r)$ of a measurable function $r:\ \Omega\to \mathbb{R}$ is defined as the union of all open sets $U$ on which the essential infimum of $|\kappa|$ is positive.

\begin{theorem}\label{thm:support_from_open_balls}
\begin{enumerate}
\item[(a)] Let $q\leq q_0$. For every open set $B\subseteq \Omega$ and every $\alpha>0$

\begin{alignat}{3}
\label{eq:supp_open_balls_a1} 
q&\leq q_0-\alpha \chi_B\quad && \text{ implies } &\quad \Lambda(q) &\leq_{d(q)} \Lambda(q_0)-\alpha\mathcal{T}_{B},\\
\label{eq:supp_open_balls_a2}
\Lambda(q) &\leq_\text{fin}  \Lambda(q_0)-\alpha\mathcal{T}_{B} \quad && \text{ implies } & \quad B&\subseteq \supp(q-q_0).
\end{alignat}

Hence,
\begin{align*}
\lefteqn{\mathrm{inn\,supp}(q-q_0)}\\
&\subseteq \bigcup \{B\subseteq \Omega \text{ open ball}:\ \exists \alpha>0: \Lambda(q) \leq_{d(q)} \Lambda(q_0)-\alpha\mathcal{T}_{B}\}\\
&\subseteq \bigcup \{B\subseteq \Omega \text{ open ball}:\ \exists \alpha>0: \Lambda(q) \leq_\text{fin} \Lambda(q_0)-\alpha\mathcal{T}_{B}\}\\
&\subseteq \supp(q-q_0).
\end{align*}
\item[(b)] Let $q\geq q_0$. For every open set $B\subseteq \Omega$ and every $\alpha>0$

\begin{alignat}{3}
\label{eq:supp_open_balls_b1} 
q&\geq q_0+\alpha \chi_B \quad && \text{ implies } &\quad 
\exists \widetilde\alpha>0:\ \Lambda(q) &\geq_\text{fin} \Lambda(q_0)+\widetilde\alpha\mathcal{T}_{B},\\
\label{eq:supp_open_balls_b2}
\Lambda(q) &\geq_\text{fin} \Lambda(q_0)+\alpha\mathcal{T}_{B} \quad && \text{ implies } & \quad 
q&\geq q_0+\alpha \chi_B.
\end{alignat}
Hence,
\begin{align*}
\mathrm{inn\,supp}(q-q_0)=\bigcup \{B\subseteq \Omega \text{ open ball}:\ \exists \alpha>0: \Lambda(q) \geq_\text{fin} \Lambda(q_0)+\alpha\mathcal{T}_{B}\}.
\end{align*}
\end{enumerate}
\end{theorem}
\begin{proof}
\begin{enumerate}
\item[(a)] If $q_1\leq q_0-\alpha \chi_B$, then we obtain using Theorem \ref{thm:converse_mon_frechet}, and Corollary \ref{remark:monotonicity_with_Frechet} that
\begin{align*}
\Lambda(q)-\Lambda(q_0)&\leq_{d(q)} \Lambda'(q_0)(q-q_0)\leq -\alpha \Lambda'(q_0)\chi_B=-\alpha \mathcal{T}_{B},
\end{align*}
so that \eqref{eq:supp_open_balls_a1} is proven. On the other hand, if $\Lambda(q) \leq_\text{fin} \Lambda(q_0)-\alpha\mathcal{T}_{B}$ then we obtain from Theorem \ref{thm:converse_mon_frechet}, and Corollary \ref{remark:monotonicity_with_Frechet}, that there exists $c>0$ with
\begin{align*}
\alpha\Lambda'(q_0)\chi_B&=
\alpha \mathcal{T}_{B}\leq_\text{fin} \Lambda(q_0)-\Lambda(q)
\leq_\text{fin} \Lambda'(q)(q_0-q)\\
&\leq \norm{q_0-q}_{L^\infty(\Omega)}\Lambda'(q)\chi_{\supp(q-q_0)}\\
&\leq_\text{fin} c \norm{q_0-q}_{L^\infty(\Omega)}\Lambda'(q_0)\chi_{\supp(q-q_0)},
\end{align*}
and that this implies 
\[
\alpha \chi_B\leq c \norm{q_0-q}_{L^\infty(\Omega)}\chi_{\supp(q-q_0)},
\]
so that \eqref{eq:supp_open_balls_a2} is proven.
\item[(b)] 
Let $q\geq q_0+\alpha \chi_B$. By Theorem~\ref{Theorem for monotonicity}, there exists a subspace
$V\subseteq H_{q_0+\alpha \chi_B}(\Omega_e)$ with $\dim(V)\leq d(q_0+\alpha \chi_B)$
so that 
\[
\langle \Lambda(q)F,F\rangle \geq \langle \Lambda(q_0+\alpha \chi_B) F, F\rangle \quad \text{ for all } 
F\in V^\perp\subseteq H_{q_0+\alpha \chi_B}(\Omega_e).
\]
Moreover, by Theorem~\ref{thm:useful} there also exists a subspace $V'\subseteq H_{q_0+\alpha \chi_B}(\Omega_e)$ with $\dim(V')\leq d(q_0)$ and a constant $c>0$ so that
\begin{align*}
\langle \left(\Lambda(q_0+\alpha \chi_B)-\Lambda(q_0)\right)F,F\rangle
&\geq \alpha \int_{B} |u_{q_0+\alpha\chi_B}|^2\dx\\
&\geq \alpha c \int_{B} |u_{q_0}|^2\dx=c\alpha \langle \mathcal{T}_{B} F,F\rangle
\end{align*}
for all $F\in V'^\perp\subseteq H_{q_0+\alpha \chi_B}$, where $u_{q_0+\alpha\chi_B}=S_{q_0+\alpha\chi_B}(F)$, and 
$u_{q_0}=S_{q_0}(F)$. Hence
\[
\langle ( \Lambda(q)-\Lambda(q_0) )F,F\rangle\geq c\alpha \mathcal{T}_{B},
\]
holds for all $F\in (V+V')^\perp\subseteq H_{q_0+\alpha \chi_B}(\Omega_e)$, which is a subspace of codimension 
$\dim(N_{q_0+\alpha\chi_B})$ in $H(\Omega_e)$. Hence,
\[
 \Lambda(q) \geq_d \Lambda(q_0)  + c\alpha \mathcal{T}_{B}\quad \text{ with } d= d(q)+d(q_0+\alpha \chi_B)+\dim(N_{q_0+\alpha\chi_B}),
\]
which shows \eqref{eq:supp_open_balls_b1}.
On the other hand, $\Lambda(q) \geq_\text{fin} \Lambda(q_0)+\alpha\mathcal{T}_{B}$ implies by Corollary \ref{remark:monotonicity_with_Frechet}
\begin{align*}
\alpha \Lambda'(q_0)\chi_B=\alpha\mathcal{T}_{B}\leq_\text{fin} \Lambda(q) - \Lambda(q_0) 
\leq_\text{fin} \Lambda'(q_0)(q-q_0),
\end{align*}
so that it follows from Theorem \ref{thm:converse_mon_frechet} that
\[
\alpha \chi_B \leq q-q_0,
\]
which proves \eqref{eq:supp_open_balls_b1}.
\end{enumerate}
\end{proof}

\section{Uniqueness and Lipschitz stability for the fractional Calder\'on problem with finitely many measurements} \label{Section 5}

In this section let $\QQ\subseteq L^\infty(\Omega)$ be a finite dimensional subspace and, with a fixed constant $a>0$, let 
\[
\QQ_{[-a,a]}:=\{ q\in \QQ:\ \norm{q}_{L^\infty(\Omega)}\leq a\}.
\] 
We will show that a sufficiently high number
of measurements of the DtN operator uniquely determines a potential in $\QQ_{[-a,a]}$ and prove a Lipschitz stability result.

To formulate our result, we denote the orthogonal projection operators from $H(\Omega_e)$ to a subspace $H$ by $P_H$,
i.e. $P_H$ is the linear operator with
\[
P_H:\ H(\Omega_e)\to H, \quad P_H F:=\left\{ \begin{array}{l l} F & \text{ if $F\in H$,}\\ 0 & \text{ if $F\in H^\perp\subseteq H(\Omega_e).$}
\end{array}\right.
\]
$P_H':\ H^*\to H(\Omega_e)^*$ denotes the dual operator of $P_H$. For possibly resonant potentials $q_1,q_2\in L^\infty(\Omega)$, the subspace $H$ might contain non-admissible Dirichlet boundary values, so we also require the orthogonal projection $P_{q_1q_2}:=P_{H_{q_1,q_2}(\Omega_e)}$.

\begin{theorem}\label{thm:stability}
For each sequence of subspaces 
\[
H_1\subseteq H_2\subseteq H_3\subseteq ...\subseteq H(\Omega_e), \quad \text{ with } \quad
\overline{\bigcup_{l\in \N} H_l}=H(\Omega_e),
\]
there exists $k\in \N$, and $c>0$, so that
\begin{align}\label{Lipschitz stability estimate}
\left\| P_{H_l}' P_{q_1q_2}' \left( \Lambda(q_2)-\Lambda(q_1) \right) P_{q_1q_2} P_{H_l}\right\|_{\LL(H(\Omega_e),H(\Omega_e)^*)}
\geq \frac{1}{c} \norm{q_2-q_1}_{L^\infty(\Omega)}
\end{align}
for all $q_1,q_2\in \QQ_{[-a,a]}$ and all $l\geq k$.
\end{theorem}

Before we prove Theorem \ref{thm:stability}, let us briefly remark on its implications for some special cases.

\begin{remark}
Theorem \ref{thm:stability} implies that there exists $c>0$ so that
\[
\norm{\Lambda(q_2)-\Lambda(q_1)}_{\LL(H_{q_1,q_2}(\Omega_e),H(\Omega_e)^*)}
\geq c \norm{q_2-q_1}_{L^\infty(\Omega)} \quad \text{for all $q_1,q_2\in \QQ_{[-a,a]}$.}
\] 

If $\{F_1,F_2,\ldots\}\subseteq H(\Omega_e)$ is a set of Dirichlet values whose linear span is dense in $H(\Omega_e)$,
then Theorem \ref{thm:stability} implies that there exists $k\in \N$, so that every non-resonant $q\in \QQ_{[-a,a]}$ is uniquely determined by the finitely many entries of the matrix
\[
A(q)=\left(\langle   \Lambda(q) F_i,  F_j\rangle\right)_{i,j=1,\ldots,k}\in \R^{k\times k}.
\]
Moreover, if $\{F_1,F_2,\ldots\}$ is an orthonormal (Schauder) basis of $H(\Omega_e)$, then there exists $k\in \N$, and $c>0$,
so that 
\[
\norm{A(q_2)-A(q_1)}_2\geq c \norm{q_2-q_1}_{L^\infty(\Omega)} \quad \text{ for all non-resonant $q\in \QQ_{[-a,a]}$,}
\]
where $\norm{A}_2$ is the spectral norm of the matrix $A\in \R^{k\times k}$.
\end{remark}

The general outline of the proof of Theorem \ref{thm:stability} is as follows. In Lemma~\ref{lemma:stability_sets}, we will 
derive a number of subsets $M_1,\ldots,M_m\subseteq \Omega$, on which normalized potential differences can be estimated from above or below. Then we define for each of these sets a special potential $\widehat q_j\in L^\infty(\Omega)$, which is large on $M_j$ and small on $\Omega\setminus M_j$, and show (in Lemma~\ref{lemma:stability_hatq}) that certain energy terms 
for the solutions for an arbitrary $q\in L^\infty(\Omega)$ can always be estimated by solutions corresponding to these special potentials $\widehat q_j$. Lemma~\ref{lemma:stability_dimbound} gives a bound on the maximal codimension of the subspaces arising from resonances, and Lemma~\ref{lemma:stability_locpot} shows the existence of sufficiently many (depending on the maximal codimension) Dirichlet boundary values $\widehat F_{ij}$ to control the energy terms arising from the special potentials $\widehat q_j$. 
The constant $c>0$ of the Lipschitz stability estimate \eqref{Lipschitz stability estimate} and the subspace index $k\in \N$ for Theorem \ref{thm:stability}, will be defined in Lemma~\ref{lemma:stability_locpot} via the maximal norm of the finitely many Dirichlet values $\widehat F_{ij}$, and the possibility of 
sufficiently well approximating $\widehat F_{ij}$ in $H_k$. Finally, we prove that Theorem \ref{thm:stability} holds with these constants $c>0$ and $k\in \N$. 

Let us stress that this construction (the sets $M_1,\ldots,M_m$, the finitely many special potentials $\widehat q_j$, the dimension bounds, the finitely many special Dirichlet data $\widehat F_{ij}$, and thus the constant $c>0$ of \eqref{Lipschitz stability estimate}, and the subspace index $k\in \N$) do only depend on the a-priori data $\QQ_{[a,b]}$ and $\Omega\subseteq \R^n$.

To motivate the first lemma, let
us note that a piecewise constant function on some partition of $\Omega$ with $L^\infty(\Omega)$-norm equal to $1$,
must be either $1$ or $-1$ on at least one of the subsets of the partition, which is a useful property for applying monotonicity estimates,
cf., e.g., \cite{harrach2019global}. The following lemma generalizes this property to our arbitrary finite-dimensional subspace $\QQ\subset L^\infty(\Omega)$.

\begin{lemma}\label{lemma:stability_sets}
Let $\QQ_1:=\{r\in \QQ:\ \norm{r}_{L^\infty(\Omega)}=1\}$.
There exists a family of measurable subsets $M_1,\ldots,M_m$, $m\in \N$,
with positive measure, so that for all $r\in \QQ_1$, there exists $j\in \{1,\ldots,m\}$ with either
$r|_{M_j}\geq \frac{1}{2}$, or $r|_{M_j}\leq -\frac{1}{2}$.
Hence, either 
\[
r\geq \frac{1}{2} \chi_{M_j} - \chi_{\Omega\setminus M_j}, \quad \text{ or } \quad -r\geq \frac{1}{2} \chi_{M_j} - \chi_{\Omega\setminus M_j}.
\]
\end{lemma}
\begin{proof}
We argue by compactness. For $r\in \QQ_1$, $\norm{r}_{L^\infty(\Omega)}=1$ implies that at least one
of the sets $r^{-1}(]\frac{1}{2},\frac{3}{2}[)$ or $r^{-1}(]-\frac{3}{2},-\frac{1}{2}[)$ must be of positive measure.
In the first case we define 
\[
M_r:=r^{-1}\left( \left] \textstyle \frac{1}{2},\frac{3}{2}\right[ \right),\quad 
\mathcal O_r:=\left\{ \widetilde r\in L^\infty(\Omega):\ \textstyle \norm{\widetilde r|_{M_r}-\chi_{M_r}}_{L^\infty(M_r)}<\frac{1}{2} \right\},
\]
and otherwise we define
\[
M_r:=r^{-1}\left( \left] \textstyle -\frac{3}{2},-\frac{1}{2}\right[ \right),\quad 
\mathcal O_r:=\left\{ \widetilde r\in L^\infty(\Omega):\ \textstyle \norm{\widetilde r|_{M_r}+\chi_{M_r}}_{L^\infty(M_r)}<\frac{1}{2} \right\}.
\]
Then $M_r$ has positive measure, $\mathcal O_r$ is an open subset of $L^\infty(\Omega)$, and $r\in \mathcal O_r$ implies that
\[
\QQ_1\subseteq \bigcup_{r\in \QQ_1} \mathcal O_r.
\]
By compactness, there exist $r_1,\ldots,r_m\in \QQ_1$ with $\QQ_1\subseteq \bigcup_{j=1,\ldots,m} \mathcal O_{r_j}$,
so that the assertion follows with $M_j:=M_{r_j}$, $j=1,\ldots,m$.
\end{proof}

We now use the idea from the constructive Lipschitz stability proof in \cite[Section 5]{harrach2019global} to
replace general potentials from $\QQ_{[-a,a]}$ by a finite number of special potentials.
\begin{lemma}\label{lemma:stability_hatq}
With the constant $a>0$ and the sets $M_1,\ldots,M_m$ from Lemma~\ref{lemma:stability_sets}, we define 
\[
\widehat q_j\in L^\infty(\Omega)\quad \text{ by } \quad \widehat q_j:=2a \chi_{M_j} - 7a \chi_{\Omega\setminus M_j}, \quad j=1,\ldots,m.
\]

If $q\in \QQ_{[-a,a]}$ and $r\in \QQ_1$ fulfills $r\geq \frac{1}{2} \chi_{M_j} - \chi_{\Omega\setminus M_j}$ with $j\in \{1,\ldots,m\}$,
then there exists a subspace $V\subseteq H_{q,\widehat q_j}(\Omega_e)$ with $\dim V \leq d(q)+d(\widehat q_j)$, so that
\[
\int_\Omega r |S_q(F)|^2\dx \geq  \int_\Omega \left( \frac{1}{6} \chi_{M_j} - \frac{4}{3} \chi_{\Omega\setminus M_j} \right) |S_{\widehat q_j}(F)|^2 \dx \quad \text{for all } F\in V^\perp\subseteq H_{q,\widehat q_j}(\Omega_e).
\]
\end{lemma}
\begin{proof}
Let $q\in \QQ_{[-a,a]}$ and $r\in \QQ_1$ fulfill $r\geq \frac{1}{2} \chi_{M_j} - \chi_{\Omega\setminus M_j}$ with $j\in \{1,\ldots,m\}$.
Then we obtain from Remark~\ref{remark_monotonicity} a subspace $V\subseteq H_{q,\widehat q_j}(\Omega_e)$ with $\dim V \leq d(q)+d(\widehat q_j)$, so that for all $F\in V^\perp\subseteq H_{q,\widehat q_j}(\Omega_e) $
\begin{align*}
\int_\Omega (\widehat q_j-q)|S_{\widehat q_j}(F)|^2 \dx\leq
\Langle \left(\Lambda(\widehat q_j)-\Lambda(q)\right) F, F \Rangle 
\leq \int_\Omega (\widehat q_j-q)|S_q(F)|^2 \dx.
\end{align*}
Observe that 
\[
a \chi_{M_j} - 8a \chi_{\Omega\setminus M_j}\leq \widehat q_j-q \leq 3a \chi_{M_j} - 6a \chi_{\Omega\setminus M_j},
\]
then it follows for all $F\in V^\perp\subseteq H_{q,\widehat q_j}(\Omega_e)$ 
\begin{align*}
\lefteqn{\quad \int_\Omega r |S_q(F)|^2\dx  \geq \int_\Omega \left( \frac{1}{2} \chi_{M_j} - \chi_{\Omega\setminus M_j}\right) |S_q(F)|^2\dx}\\
&= \frac{1}{6a} \int_\Omega \left( 3a \chi_{M_j} - 6a \chi_{\Omega\setminus M_j}\right) |S_q(F)|^2\dx
\geq \frac{1}{6a} \int_\Omega \left( \widehat q_j - q \right) |S_q(F)|^2\dx\\
&\geq \frac{1}{6a}  \int_\Omega (\widehat q_j-q)|S_{\widehat q_j}(F)|^2 \dx
 \geq \frac{1}{6a}  \int_\Omega \left(a \chi_{M_j} - 8a \chi_{\Omega\setminus M_j} \right) |S_{\widehat q_j}(F)|^2 \dx\\
& =   \int_\Omega \left( \frac{1}{6} \chi_{M_j} - \frac{4}{3} \chi_{\Omega\setminus M_j} \right) |S_{\widehat q_j}(F)|^2 \dx.
\end{align*}
\end{proof}

The next lemma shows that the codimension of the subspaces where the DtN operators are defined, and 
the subspaces where the monotonicity relations hold, can be uniformly bounded in $\QQ_{[-a,a]}$.
\begin{lemma}\label{lemma:stability_dimbound}
There exists numbers $d,N\in \N$, so that
\[
\dim (N_{q})\leq N \quad \text{ and } \quad d(q)\leq d \quad \text{ for all } q\in \QQ_{[-a,a]},
\]
where $N_q$ is defined by \eqref{N_q} and $d(q)$ is given by Definition \ref{def:dimension_d}.
\end{lemma}
\begin{proof}
The first assertion follows from Theorem \ref{thm:dimensions}(b) with a standard compactness argument. 
The second assertion follows from Theorem \ref{thm:dimensions}(a) with $d:=d(-a)$, where $d(-a)$ is the number defined in Definition~\ref{def:dimension_d} for $q\equiv -a$.
\end{proof}

Our last lemma demonstrates how to control the energy terms in Lemma \ref{lemma:stability_hatq}, and defines the Lipschitz
stability constant $c>0$ and the subspace index $k\in \N$, with which the assertion of Theorem \ref{thm:stability} holds.
\begin{lemma}\label{lemma:stability_locpot}
	Let $d,N\in \N$ be the numbers given in Lemma \ref{lemma:stability_dimbound}, then we have 
\begin{enumerate}[(a)]
\item For all $j\in \{1,\ldots,m\}$, there exist Dirichlet data $\widehat F_{i,j}\in H_{\widehat q_j}(\Omega_e)$ with
\begin{align}
\label{eq:stability_locpot_hatF}
\int_\Omega \left( \frac{1}{6} \chi_{M_j} - \frac{4}{3} \chi_{\Omega\setminus M_j} \right) |S_{\widehat q_j}(\widehat F_{i,j})|^2 \dx  &\geq 2,\\
\label{eq:stability_locpot_hatF_ortho1}
\int_\Omega \left( \frac{1}{6} \chi_{M_j} - \frac{4}{3} \chi_{\Omega\setminus M_j} \right) S_{\widehat q_j}(\widehat F_{i,j}) S_{\widehat q_j}(\widehat F_{i',j}) \dx  &= 0,\\
\label{eq:stability_locpot_hatF_ortho2}
\left( \widehat F_{i,j}, \widehat F_{i',j} \right)_{H(\Omega_e)}  &= 0,
\end{align}
for all $i,i'=1,\ldots, 3d+2N+1$ with $i'\neq i$.
We set
\[
c:=2 \max\left\{ \norm{\widehat F_{i,j}}_{H(\Omega_e)}^2:\ j=1,\ldots,m,\ i=1,\ldots, 3d+2N+1\right\}.
\]
\item For $\delta:=\frac{1}{3d+2n+2}$, and for each sequence of subspaces 
\[
H_1\subseteq H_2\subseteq H_3\subseteq ...\subseteq H(\Omega_e), \quad \text{ with } \quad
\overline{\bigcup_{l\in \N} H_l}=H(\Omega_e),
\]
there exists $k\in \N$, and $F_{i,j}\in H_k\cap H_{\widehat q_j}(\Omega_e)$, so that
\begin{align}
\label{eq:stability_locpot_Fij}
\int_\Omega \left( \frac{1}{6} \chi_{M_j} - \frac{4}{3} \chi_{\Omega\setminus M_j} \right) |S_{\widehat q_j}(F_{i,j})|^2 \dx  &\geq 2-\delta,\\
\label{eq:stability_locpot_Fij_ortho1}
\left|\int_\Omega \left( \frac{1}{6} \chi_{M_j} - \frac{4}{3} \chi_{\Omega\setminus M_j} \right) S_{\widehat q_j}(F_{i,j}) S_{\widehat q_j}(F_{i',j}) \dx\right| &\leq \delta,\\
\label{eq:stability_locpot_Fij_ortho2}
\left| \left(  F_{i,j}, F_{i',j} \right)_{H(\Omega_e)} \right|  &\leq  \frac{c}{2}\delta,
\end{align}
and $\norm{F_{i,j}}_{H(\Omega_e)}^2\leq (1+\delta)\frac{c}{2}$
for all $j=1,\ldots,m$, and all $i,i'=1,\ldots, 3d+2N+1$ with $i'\neq i$.
\item For all $j=1,\ldots,m$, all subspaces $V\subseteq H_{\widehat q_j}(\Omega_e)$, with $\dim V \leq 3d+2N$, contain an
element $F_j\in V^\perp \cap H_k$ with
\[
\int_\Omega \left( \frac{1}{6} \chi_{M_j} - \frac{4}{3} \chi_{\Omega\setminus M_j} \right) |S_{\widehat q_j}(F_{j})|^2 \dx \geq 1,
\quad \text{ and } \quad \norm{F_{j}}_{H(\Omega_e)}^2\leq c.
\]
\end{enumerate}
\end{lemma}
\begin{proof} 
Let $j\in \{1,\ldots,m\}$.
\begin{enumerate}[(a)]
\item Theorem~\ref{thm:localized} yields that every subspace $V^\perp$ of finite codimension in $H_{\widehat q_j}(\Omega_e)$ contains 
$F$ that fulfill the property \eqref{eq:stability_locpot_hatF}. Hence, for $i=1$, we can apply Theorem~\ref{thm:localized} on $H_{\widehat q_j}(\Omega_e)$ 
to obtain $\widehat F_{1,j}$, and for $i>1$, we obtain $\widehat F_{i,j}$ by applying Theorem~\ref{thm:localized} on the subspace
\begin{align*}
& \left\{\widehat  F\in H_{\widehat q_j}(\Omega_e):\ 
\int_\Omega \left( \frac{1}{6} \chi_{M_j} - \frac{4}{3} \chi_{\Omega\setminus M_j} \right) S_{\widehat q_j}(\widehat F) S_{\widehat q_j}(\widehat F_{i',j}) \dx  = 0,\right.\\
& \qquad \qquad \qquad \qquad \left. \text{ and } \left( \widehat F, \widehat F_{i',j} \right)_{H(\Omega_e)}=0
\text{ for all } i'\in\{1,\ldots,i-1\}\right\},
\end{align*}
which is obviously of finite codimension in $H_{\widehat q_j}(\Omega_e)$, and this shows \eqref{eq:stability_locpot_hatF_ortho1} and \eqref{eq:stability_locpot_hatF_ortho2}.
\item From the finite codimension of $H_{\widehat q_j}(\Omega_e)$ in $H(\Omega_e)$, we obtain that
$\bigcup_{l\in \N} H_l\cap H_{\widehat q_j}(\Omega_e)$ is dense in $H_{\widehat q_j}(\Omega_e)$.
Hence, the assertion (b) follows from the continuity of the solution operator $S_{\widehat q_j}$.
\item Since $V\subseteq H_{\widehat q_j}(\Omega_e)$ has $\dim V \leq 3d+2N$, there exists a non-trivial linear combination
\[
0\neq F_j:=\sum_{i=1}^{3d+2N+1} \lambda_i F_{i,j} \in V^\perp \cap H_k, \quad \text{ with coefficient } \quad \lambda_i\in \R,
\]
where we normalize the coefficients so that $\sum_{i=1}^{3d+2N+1} |\lambda_i|^2 =1$ and $k\in \N$ is the same number given as in (b). Then,
\[
\sum_{i,i'=1}^{3d+2N+1} |\lambda_i|\, |\lambda_{i'}| \leq 3d+2N+1.
\]
By using \eqref{eq:stability_locpot_Fij}, \eqref{eq:stability_locpot_Fij_ortho1} and \eqref{eq:stability_locpot_Fij_ortho2}, a simple calculation shows that
\begin{align*}
\int_\Omega \left( \frac{1}{6} \chi_{M_j} - \frac{4}{3} \chi_{\Omega\setminus M_j} \right) |S_{\widehat q_j}(F_{j})|^2 \dx
&\geq 2-(3d+2N+2)\delta=1,\\
\norm{F_{j}}_{H(\Omega_e)}^2 &\leq \left( 1+(3d+2N+2)\delta \right) \frac{c}{2}=c.
\end{align*}
\kommentar{
\begin{align*}
\lefteqn{\int_\Omega \left( \frac{1}{6} \chi_{M_j} - \frac{4}{3} \chi_{\Omega\setminus M_j} \right) |S_{\widehat q_j}(F_{j})|^2 \dx}\\
&= \sum_{i,i'=1}^{3d+2N+1} \lambda_i \lambda_{i'} \int_\Omega \left( \frac{1}{6} \chi_{M_j} - \frac{4}{3} \chi_{\Omega\setminus M_j} \right) S_{\widehat q_j}(F_{i,j})S_{\widehat q_j}(F_{i',j}) \dx\\
&= \sum_{i=1}^{3d+2N+1} \lambda_i^2 \int_\Omega \left( \frac{1}{6} \chi_{M_j} - \frac{4}{3} \chi_{\Omega\setminus M_j} \right) |S_{\widehat q_j}(F_{i,j})|^2 \dx\\
&\quad {} + \sum_{i=1}^{3d+2N+1} \sum_{i\neq i'=1}^{3d+2N+1} \lambda_i \lambda_{i'} \int_\Omega \left( \frac{1}{6} \chi_{M_j} - \frac{4}{3} \chi_{\Omega\setminus M_j} \right) S_{\widehat q_j}(F_{i,j})S_{\widehat q_j}(F_{i',j}) \dx\\
&\geq \sum_{i=1}^{3d+2N+1} \lambda_i^2  (2-\delta) - \sum_{i,i'=1}^{3d+2N+1}  |\lambda_i|\, |\lambda_{i'}| \delta\\
&\geq 2-\delta - (3d+2N+1)\delta=2-(3d+2N+2)\delta
\end{align*}
and
\begin{align*}
\norm{F_{j}}_{H(\Omega_e)}^2&= \sum_{i=1}^{3d+2N+1} \lambda_i^2 \norm{F_{i,j}}_{H(\Omega_e)}^2
+ \sum_{i=1}^{3d+2N+1} \sum_{i\neq i'=1}^{3d+2N+1} \lambda_i \lambda_{i'} \left( F_{i,j}, F_{i',j} \right)_{H(\Omega_e)}\\
&\leq (1+\delta)\frac{c}{2}+ (3d+2N+1)\delta \frac{c}{2}=\left( 1+(3d+2N+2)\delta \right) \frac{c}{2}.
\end{align*}
}
\end{enumerate}
This completes the proof.
\end{proof}

Now, we can prove Theorem \ref{thm:stability}.

\emph{Proof of Theorem \ref{thm:stability}.} 
Let $q_1,q_2\in \QQ_{[-a,a]}$ with $q_1\neq q_2$, and set $r:=\frac{q_2-q_1}{\norm{q_2-q_1}_{L^\infty(\Omega)}}$. Then, by Lemma~\ref{lemma:stability_sets}, there exist 
$j\in \{1,\ldots,m\}$ with either 
\[
\text{(a) }
r\geq \frac{1}{2} \chi_{M_j} - \chi_{\Omega\setminus M_j}, \quad \text{ or } \quad \text{(b) } -r\geq \frac{1}{2} \chi_{M_j} - \chi_{\Omega\setminus M_j}.
\] 
In case (a), Theorem~\ref{Theorem for monotonicity} yields that there exists a subspace $V'\subseteq H_{q_1,q_2}(\Omega_e)$ of dimension $d(q_1)$, so that
\[
\frac{\Langle \left(\Lambda(q_2)-\Lambda(q_1)\right) F, F \Rangle}{\norm{q_2-q_1}_{L^\infty(\Omega)}}
\geq \int_\Omega r |S_{q_2}(F)|^2 \dx\quad \text{ for all } F\in (V')^\perp\subseteq H_{q_1,q_2}(\Omega_e).
\]
Also, Lemma~\ref{lemma:stability_hatq} yields a subspace $V''\subseteq H_{q_2,\widehat q_j}(\Omega_e)$ with $\dim V'' \leq d(q_2)+d(\widehat q_j)$, so that
\[
\int_\Omega r |S_{q_2}(F)|^2\dx \geq  \int_\Omega \left( \frac{1}{6} \chi_{M_j} - \frac{4}{3} \chi_{\Omega\setminus M_j} \right) |S_{\widehat q_j}(F)|^2 \dx  \quad \forall F\in (V'')^\perp\subseteq H_{q_2,\widehat q_j}(\Omega_e).
\]
Then $V:=V'+V''+H_{q_1}(\Omega_e)^\perp+H_{q_2}(\Omega_e)^\perp$ is a subspace with $\dim V \leq 3d+2N$,
and, by Lemma~\ref{lemma:stability_locpot}(c), there exists $F_j\in V^\perp \cap  H_k$ with $\norm{F_j}_{H(\Omega_e)}^2\leq c$, and 
\[
\int_\Omega r |S_{q_2}(F_j)|^2\dx \geq  \int_\Omega \left( \frac{1}{6} \chi_{M_j} - \frac{4}{3} \chi_{\Omega\setminus M_j} \right) |S_{\widehat q_j}(F_j)|^2 \dx\geq 1.
\]

Since $F_j\in V^\perp\cap H_k$, and the definition of $V$ implies that $V^\perp\subseteq H_{\widehat q_j}(\Omega_e)$ is a subspace of $H_{q_1,q_2}(\Omega_e)$, we have that $P_{q_1q_2} P_{H_l} F_j=P_{q_1q_2}  F_j=F_j$ for all $l\geq k$.
Hence, it follows from the self-adjointness of $P_{H_l}' P_{q_1q_2}' \left( \Lambda(q_2)-\Lambda(q_1) \right) P_{q_1q_2} P_{H_l}$
that for all $l\geq k$,
\begin{align*}
\lefteqn{\frac{\norm{P_{H_l}' P_{q_1q_2}' \left( \Lambda(q_2)-\Lambda(q_1) \right) P_{q_1q_2} P_{H_l}}_{\LL(H(\Omega_e),H(\Omega_e)^*)}}{\norm{q_2-q_1}_{L^\infty(\Omega)}}}\\
&=  \sup_{0\neq  F\in H(\Omega_e)}
\frac{\left|\Langle \left(\Lambda(q_2)-\Lambda(q_1)\right) P_{q_1q_2} P_{H_l} F,\ P_{q_1q_2} P_{H_l} F \Rangle \right|}{\norm{q_2-q_1}_{L^\infty(\Omega)}\, \norm{F}_{H(\Omega_e)}^2\, }\\
&\geq 
\frac{\left|\Langle \left(\Lambda(q_2)-\Lambda(q_1)\right) F_j, F_j \Rangle \right|}{\norm{q_2-q_1}_{L^\infty(\Omega)}\,\norm{F_j}_{H(\Omega_e)}^2\, }
\geq \frac{1}{\norm{F_j}^2_{H(\Omega_e)}} 
\int_\Omega r |S_{q_2}(F_j)|^2\dx\\
&\geq \frac{1}{\norm{F_j}^2_{H(\Omega_e)}}  \int_\Omega \left( \frac{1}{6} \chi_{M_j} - \frac{4}{3} \chi_{\Omega\setminus M_j} \right) |S_{\widehat q_j}(F_j)|^2 \dx 
\geq \frac{1}{\norm{F_j}^2_{H(\Omega_e)}} \geq \frac{1}{c}.
\end{align*}

In case (b), Theorem~\ref{Theorem for monotonicity} yields that there exists a subspace $V'\subseteq H_{q_1,q_2}(\Omega_e)$ with dimension $d(q_2)$, so that
\[
\frac{\left| \Langle \left(\Lambda(q_2)-\Lambda(q_1)\right) F, F \Rangle \right|}{\norm{q_2-q_1}_{L^\infty(\Omega)}}
\geq - \int_\Omega r |S_{q_1}(F)|^2 \dx\quad \text{ for all } F\in (V')^\perp\subseteq H_{q_1,q_2}(\Omega_e),
\]
so that the assertion follows analogously by using Lemma~\ref{lemma:stability_hatq} with $-r$ instead of $r$.
\endproof

\section*{Acknowledgment}
The authors would like to thank Professor Mikko Salo for fruitful discussions and helpful suggestions to improve this work.

Y.-H. Lin was supported by the Finnish Centre of Excellence in Inverse Modelling and Imaging (Academy of Finland grant 284715) and also by the Academy of Finland project number 309963.

\bibliographystyle{abbrv}
\bibliography{ref}

\end{document}